\documentclass[11pt,letterpaper]{article}

\usepackage{amssymb,amsthm,amsmath}
\usepackage[pdftex]{graphicx}
\usepackage{graphicx}
\usepackage{fullpage}
\usepackage[margin=1in]{geometry}
\usepackage{algorithm,algorithmic}
\usepackage{color}
\usepackage{subfigure}

\newtheorem{theorem}{Theorem}[section]
\newtheorem{lemma}[theorem]{Lemma}
\newtheorem{proposition}[theorem]{Proposition}

\newtheorem{corollary}[theorem]{Corollary}
\newtheorem{result}{Result}

\theoremstyle{definition}

\theoremstyle{remark}

\newcommand{\R}{\mathbb{R}}

\newcommand{\la}{\langle}
\newcommand{\ra}{\rangle}
\newcommand{\matrixones}{J}

\newcommand{\bone}{\mathbf{1}}
\newcommand{\be}{\mathbf{e}}
\renewcommand{\bf}{\mathbf{f}}
\newcommand{\bg}{\mathbf{g}}

\newcommand{\bv}{\mathbf{v}}

\newcommand{\bx}{\mathbf{x}}
\newcommand{\by}{\mathbf{y}}
\newcommand{\bz}{\mathbf{z}}

\newcommand{\ks}{k^\star}
\newcommand{\zz}{Z^\star}
\newcommand{\ac}{A^\star}
\newcommand{\xc}{X^\star}
\newcommand{\zc}{Z'}
\newcommand{\ccs}{\mathcal{C}^\star} 

\newcommand{\lt}{\vartheta}
\newcommand{\V}{\mathcal{V}}
\newcommand{\E}{\mathcal{E}}

\newcommand{\range}{\mathrm{range}}
\newcommand{\lspan}{\mathrm{span}}
\newcommand{\cscc}{c\text{-}{\rm SCC}}

\definecolor{gray}{rgb}{0.8, 0.8, 0.8}

\newcommand{\bal}{\begin{aligned}}

\newcommand{\eal}{\end{aligned}}
\newcommand{\beq}{\begin{equation}}
\newcommand{\eeq}{\end{equation}}

\title{The Lov\'asz Theta Function for Recovering Planted Clique Covers and Graph Colorings} 

\author{
    Jiaxin Hou\textsuperscript{1} \and 
    Yong Sheng Soh\textsuperscript{2} \and 
    Antonios Varvitsiotis \textsuperscript{3}
}
\date{
    \textsuperscript{1}Department of Industrial and Systems Engineering, University of Southern California \\
    \textsuperscript{2} Department of Mathematics, National University of Singapore \\
    \textsuperscript{3}Engineering Systems and Design Pillar, Singapore University of Technology and Design \\ Archimedes/Athena RC, Greece \\ Centre for Quantum Technologies, National University of Singapore \\    \texttt{jiaxinho@usc.edu},   \texttt{matsys@nus.edu.sg}, \texttt{antonios@sutd.edu.sg}
}

\begin{document}

\maketitle

\begin{abstract}
The problems of computing graph colorings and clique covers are central challenges in combinatorial optimization.  Both of these are known to be NP-hard, and thus computationally intractable in the worst-case instance.  A prominent approach for computing approximate solutions to these problems is the celebrated Lov\'asz theta function $\vartheta(G)$, which is specified as the solution of a semidefinite program (SDP), and hence tractable to compute.  In this work, we move beyond the worst-case analysis and set out to understand whether the Lov\'asz theta function recovers clique covers for random instances that have a latent clique cover structure, possibly obscured by noise.  We answer this question in the affirmative and show that for graphs generated from the planted clique model we introduce in this work, the SDP formulation of $\vartheta(G)$ has a unique solution that reveals the underlying clique-cover structure with high-probability.  The main technical step is an intermediate result where we prove a deterministic condition of recovery based on an appropriate notion of sparsity.  
\end{abstract}

\noindent \emph{Keywords}: clique cover, graph coloring,  clustering, community detection, Lov\'asz theta function, beyond worst-case analysis

\section{Introduction}

Graph colorings and clique covers are central problems in combinatorial optimization.  Given an undirected graph $G = (\V,\E)$, the objective of the graph coloring problem is to partition the vertices into independent subsets; that is, all the vertices from the same subset are not adjacent.  The minimum number of  independent sets required is known as the {\em chromatic number} of $G$ and is denoted by $\chi(G)$.  The objective of the clique covering problem is to partition the vertices $\V$ such that each subset is a clique in $G$; that is, all pairs of vertices from the same set are adjacent.  The minimum number of cliques required is known as the {\em clique cover number} of $G$ and is denoted by~$\overline{\chi}(G)$.  

Graph colorings and clique covers are complementary notions -- the complement of an independent set is a clique.  In particular, a partition of $G$ into $k$ independent sets corresponds to a cover with $k$ cliques in the complementary graph $\overline{G}$.  Both clique covers and graph colorings have applications in various fields such as scheduling \cite{marx} and clustering.

The computational tasks of finding clique covers as well as graph colorings are not known to be tractable.  For instance, the problem of deciding whether a graph admits a coloring with $k$ colors is NP-complete; in fact, it is one of Karp's original 21~problems.   However, the notion of NP-hardness is rooted in a worst-case analysis perspective, which may not be representative of the instances that are typically encountered in practice.  As such, it is natural to move beyond the worst-case analysis \cite{rough} and study these problems in a suitably defined average-case instance~\cite{Ban:16}.  This is the viewpoint we adopt in this work.  We focus on the clique cover problem, which by complementation, is equivalent to graph coloring.  {The main question we wish to address is:} {\em Can the clique cover problem be efficiently solved for randomly generated instances where a clique cover structure naturally emerges?}
 
\paragraph{Planted clique cover models.}  The first step is to provide a suitable definition of a randomly generated instance.  In this paper, we generate {\em planted clique cover} instances specified by a set of vertices $\V$, a corresponding partition of the vertices $\{\ccs_l\}_{l=1}^{\ks} \subset \V$, and a parameter $p \in [0,1]$.  We generate a graph from the planted clique model by creating cliques for every subset $\ccs_l$, $1 \leq l \leq \ks$, and by subsequently introducing edges between vertices $i,j$ belonging to different cliques, with probability $p$ independently across all edges; see Figure~\ref{cluster} for an example.   

\begin{figure}[h]\label{cluster}
	\centering
     \includegraphics[width=0.3\linewidth]{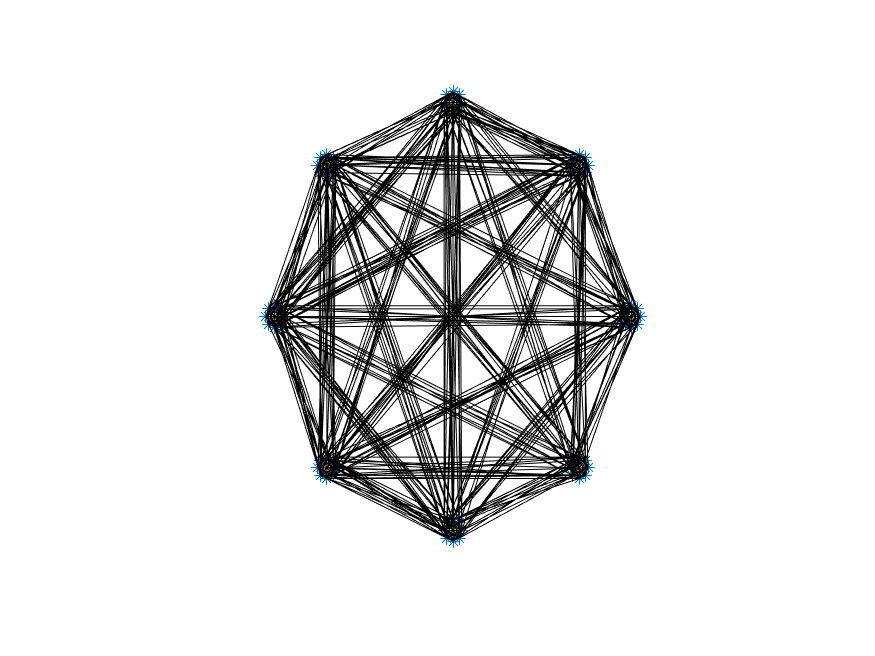}
     \includegraphics[width=0.3\linewidth]{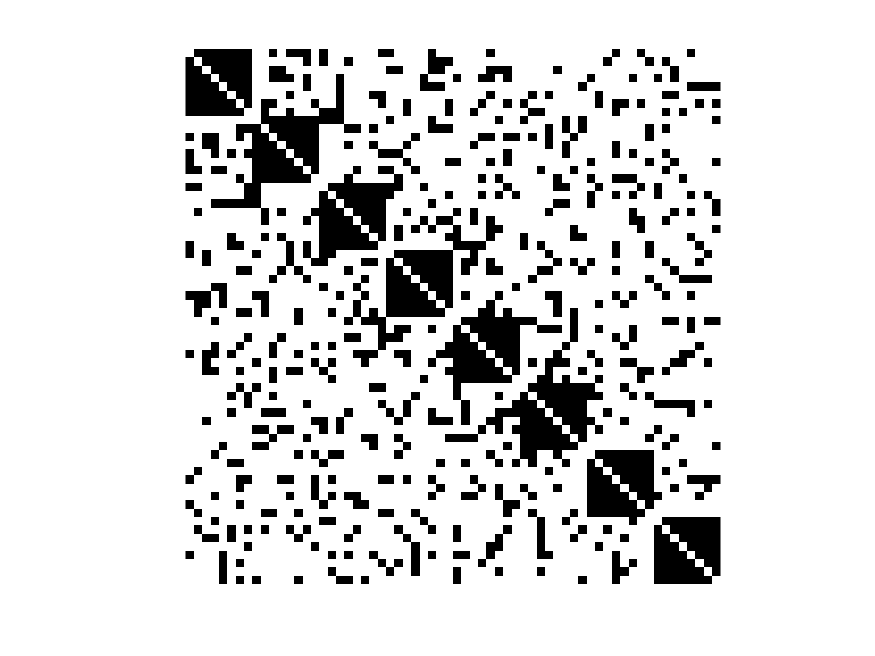}
     \caption{A planted clique instance with $8$ cliques of size $8$ each, and with parameter $p=0.2$.}
\end{figure}

\paragraph{Lov\'asz theta function.}  Graphs generated from the planted clique cover model have a natural clique cover structure obscured by noise.  Our goal is to understand whether  this latent structure can be revealed using a widely studied graph theoretic quantity.  More specifically, the Lov\'asz theta function of an undirected graph $G=(\V,\E)$ is given by
\beq\tag{P} \label{eq:lovasz_lambdamax}
\lt(G) ~~ := ~~ \min_{t,A} \Big\{t : tI+A-\matrixones\succeq 0, \ A_{i,i}=0, \ A_{i,j}=0 \text{ for all } (i,j)\not \in \E \Big\},
\eeq
where $\matrixones$ is the matrix of all ones and $I$ is the identity matrix of size $|\V|$.  The Lov\'asz theta function can be viewed as a convex relaxation of the integer programming formulation of the clique cover number.  We point out that the specific formulation \eqref{eq:lovasz_lambdamax} is simply one among a number of equivalent formulations, and we refer the interested reader to \cite{knuth} for a list of alternative definitions.  
 
An important property of $\vartheta(G)$ -- often referred to as the ``sandwich theorem'' -- is that it is a lower bound to the clique cover number and an upper bound to the stability~{number \cite{theta}}
\beq \label{sand}
\alpha(G) = \omega(\overline{G}) \leq \lt (G) \leq \overline{\chi}(G).
\eeq
Here, $\alpha(G)$ denotes the size of the largest subset of vertices of $G$ that are not adjacent, which, by taking complements, is equal to the size of the largest clique in $\overline{G}$, denoted by $\omega(\overline{G})$.  The Lov\'asz theta function is the solution of a semidefinite program (SDP), and hence computable in polynomial time \cite{NesNem:94,Ren:01} -- this stands in sharp contrast with the stability number and the chromatic number of a graph, both of which are NP-hard to compute.  In particular, the sandwich theorem has important consequences for {\em perfect} graphs -- these are graphs for which, in every induced subgraph of $G$, the chromatic number equals the size of the maximum clique.  For such graphs, one has $\omega(G) = \lt(\overline{G}) = \chi(G)$, and thus the clique number and the chromatic number of any perfect graph can be computed efficiently in polynomial time.  One can also obtain the minimal clique covers and colorings for such graphs tractably -- see, e.g., \cite[Section 6.3.2]{monique}.  More broadly, the Lov\'asz theta function has important connections to a number of combinatorial and information theoretic properties of $G$, and one of the aims our work is to draw attention to the fact that the Lov\'asz theta function is also a powerful convex relaxation for recovering clique covers.

\paragraph{Related work.}  Our work is closely related to the literature on clustering and community detection problems.  Concretely, consider a network where the presence of an edge (or possibly the size of its edge weight) models the strength of association between two nodes.  The graph clustering task concerns grouping these nodes into subsets -- often called {\em communities} -- such that the strength of association between pairs of nodes from the same subset is significantly stronger that the strength of association between pairs belonging to different subsets.  The community detection task finds application in various fields such as sociology, physics, epidemiology, and more \cite{For:10,New:06}. The term {\em community detection} is perhaps often interchangeably used with {\em graph clustering}, and in what follows we refer to this body of work by community detection.

The {\em Stochastic Block Model (SBM)} is a generative model for creating networks with pre-defined community structures studied within the community detection literature. In the two-community SBM, one generates associations between nodes from the same community with probability $q$, and associations between nodes from different communities with  probability $p<q$.  The SBM provides us with a concrete approach to understanding the performance of community detection algorithms on random instances, as opposed to focusing on worst-case instances, which can be overly conservative.  The planted clique cover model we consider in this work is a special instance of the SBM with $q=1$.

There is a substantial body of prior work on the recovery of communities under the SBM; e.g. see the survey \cite{survey} and references therein.  While there are many different families of techniques to recover these communities, the approaches most relevant to our work are those based on SDP relaxations.  As we discuss below, these techniques can also be adapted to find clique covers, but they often operate under more stringent assumptions about the data than is required in this work.

First, the seminal work by Abbe, Bandeira, and Hall \cite{ABH:16} study a SDP relaxation of a quadratic integer program that expresses the maximum likelihood estimator.  In particular, the authors focus on understanding fundamental thresholds under which recovery is possible.  Concretely, supposing that there are two communities of equal size $n$, and one sets the parameters $q = \alpha (\log n/n)$ and $p = \beta (\log n/n)$, the authors in \cite{ABH:16} study the range of values of $\alpha$ and $\beta$ for which recovery is possible or impossible.  Our SDP formulation differs from the one in \cite{ABH:16} in that ours is applicable beyond the two-clique settings they focus on.

Subsequent work by Amini and Levina \cite{amini} as well as work by Chen and Xu \cite{CX} study SDP relaxations that apply to settings with more than two clusters.  These relaxations are derived from exact  formulations  that use  binary vectors to encode membership within a cluster.  We point out that in comparison to our work, these relaxations assume additional knowledge of the data we do not.  For instance, the formulation (see SDP1 and SDP2) and the subsequent analysis (see e.g. Theorems 4.1 and 4.2) in \cite{amini} assumes knowledge of the number of clusters $k$, as well as that each cluster has the {\em same} number of elements $n$.  We point out that there is also a different SDP formulation by Oymak and Hassibi \cite{OH:11} that does not make assumptions about the cluster size; however our experiments show that  the quality of the recovery appears substantially weaker than the Lov\'{a}sz theta function.  We discuss the latter point in more detail in Section \ref{subsec:related work}.

Finally, we discuss related work by Ames and Vavasis \cite{AV:14}, who explored an SDP relaxation for the \(k\)-disjoint-clique problem aimed at finding \(k\) disjoint cliques that cover the maximum number of nodes. Our work differs  from theirs for two primary reasons. First, the cliques identified in their approach are not required to cover all nodes; they may only cover a subset, and thus they are not feasible for the clique cover problems. Second, the formulation in \cite{AV:14} requires specifying the number of cliques \(k\), unlike the Lovász theta number which does not need on such assumptions. We provide a numerical comparison between these methodologies in Section \ref{subsec:related work}.

\paragraph{Recovery of planted clique covers via $\vartheta(G)$.} The goal of this work is to show that the Lov\'asz theta recovers the planted clique cover model with high probability.  What does such a claim entail?  First, note that the latter inequality in \eqref{sand} -- namely $ \vartheta (G) \leq \overline{\chi}(G)$ -- relies on the fact that any clique cover corresponds to a feasible solution of the Lov\'asz theta formulation \eqref{eq:lovasz_lambdamax}.  Concretely, given a graph $G=(\V,\E)$, let $\mathcal{C}=\{\mathcal{C}_l\}_{l=1}^{k}$ be any clique cover.  Consider the following matrix
\begin{equation} \label{eq:feasiblesolution}
X = kI + \Big(\sum_{l=1}^kk\bone_{\mathcal{C}_l}\bone_{\mathcal{C}_l}^T-kI \Big) - \matrixones = {1\over k}\sum_{l=1}^{k} (k\bone_{\mathcal{C}_l}-\be) (k\bone_{\mathcal{C}_l}-\be)^T.
\end{equation}
Here, we use $\bone_{\mathcal{C}} \in \mathbb{R}^{|V|}$ to denote the binary indicator variable whose $i$-th coordinate is one precisely when $i\in \mathcal{C}$.  It is evident that $X \succeq 0$, and hence $A=\sum_{l=1}^kk\bone_{\mathcal{C}_l}\bone_{\mathcal{C}_l}^T-kI$ is a feasible solution to \eqref{eq:lovasz_lambdamax} with objective value $t=k$.  In what follows, we define $\ac$ and $\xc$ with respect to the planted clique covering $\{\ccs_l\}_{l=1}^{\ks}$
\begin{equation} \label{eq:feasiblesolution}
\ac := \sum_{l=1}^{\ks} \ks \bone_{\ccs_l}\bone_{\ccs_l}^T-\ks I \quad \text{ and } \quad  \xc  := \ks I + \ac - \matrixones = {1\over \ks}\sum_{l=1}^{\ks} (\ks\bone_{\ccs_l}-\be) (\ks\bone_{\ccs_l}-\be)^T.
\end{equation}
When we say that we wish to show that the Lov\'asz theta recovers the planted clique, we specifically mean to prove that the pair $(\ac,\ks)$ is the {\em unique} optimal solution to \eqref{eq:lovasz_lambdamax} for graphs generated according to the above planted clique cover model, with high probability.     

Our first main result is as follows.

\begin{result} \label{thm:result1}
Let $G$ be a random planted clique cover instance  defined on the cliques $\{ \ccs_l \}_{l=1}^{\ks}$, and where we introduce edges between cliques with probability $p < c:= \min \big\{ \frac{1}{4} (\frac{\min_l 1/|\ccs_l|}{\sum_l 1/|\ccs_l|})^2,{ 1\over 100 }\big\}$.  Then $(\ac,\ks)$ is the unique optimal solution to  the Lov\'asz theta formulation \eqref{eq:lovasz_lambdamax} with probability greater~than 
$$1 - \sum_{i=1}^{\ks} {|\ccs_i| \sum_{j \neq i}\exp(-2|\ccs_j|(c-p)^2}).
$$
\end{result}
The simplest type of graphs in the context of the clique cover problem  are disjoint unions of cliques.  An important intermediate step of our proof is to show that the Lov\'asz theta function succeeds at recovering clique covers for graphs that resemble disjoint union of cliques.  More concretely, we say that a graph $G$ with clique cover $\{ \ccs_l \}_{l=1}^{\ks}$ satisfies the $c$-sparse clique cover (c-SCC) property if, for every vertex $v$ and every clique $\ccs_l$ that does not contain $v$, one has
\begin{equation} \label{eq:neighborcondition}\tag{$c$-SCC}
|e(v, \ccs_l)| \leq c | \ccs_l|. 
\end{equation}
Here, given disjoint subsets of vertices $\mathcal{S}$ and $\mathcal{T}$, we let $e(\mathcal{S},\mathcal{T}) \subset \E$ denote the edges in $\E$ with one vertex in $\mathcal{S}$ and another vertex in $\mathcal{T}$.  Note that if the graph $G$ satisfies the $\cscc$ property for some $c<1$, then every vertex $v \in \V$ cannot form a clique with any other clique $\ccs_l$ for which it does not belong to.

\begin{result} \label{thm:result2}
Suppose $G$ is a graph that satisfies the $c$-SCC property for some 
$c < \min \{ \frac{1}{4} (\frac{\min_l 1/|\ccs_l|}{\sum_l 1/|\ccs_l|})^2,\frac{1}{100} \}$.  Then $\xc$ is the unique optimal solution to the Lov\'asz theta formulation \eqref{eq:lovasz_lambdamax}.
\end{result}

If the cliques $\{ \ccs_l \}_{l=1}^{\ks}$ are of equal size, then the above condition requires $c \lesssim (1/\ks)^2$.  An important observation from Result \ref{thm:result2} is that it reveals the types of graphs for which the Lov\'asz theta function is most effective at uncovering the underlying clique cover -- these are graphs for which the $c$-neighborly parameter is small.  The parameter $p$ in our model is related to the neighborly parameter $c$ in that small choices of $p$ will generate graphs with small neighborly parameter.  In our numerical experiments in Section \ref{sec:numerical}, we see that the parameter $p$ (and hence, by extension $c$) also reveal something about the difficulty of computing the clique covering number of a particular graph -- the ``hard'' instances (these are defined by those requiring a large number of simplex solves using a MILP solver) correspond to graphs generated using intermediate values of $p$ (say $p \in [ 0.6, 0.7]$).

\subsection{Proof Outline} \label{sec:outline}

In this section, we give an outline of the proof of our main results.  The key part of our argument is to produce a suitable dual optimal solution that certifies the optimality as well as the uniqueness of $(\ac,\ks)$ in \eqref{eq:lovasz_lambdamax}.  First, as a reminder, the dual program to \eqref{eq:lovasz_lambdamax} is the following semidefinite~program
\beq\label{eq:dual}\tag{D}
\max \Big\{ \, \langle \matrixones , Z \rangle \,:\, \mathrm{tr}(Z) = 1, Z_{i,j} = 0 \text{ for all } {(i,j) \in \E}, \ Z \succeq 0  \Big\}.
\eeq

A basic consequence of weak duality is that, in order to show $(\ac,\ks)$ is {\em an} optimal solution to~\eqref{eq:lovasz_lambdamax} (note that this is equivalent to $\vartheta(G) = \ks$), it suffices to produce a dual feasible solution that attains the same objective; i.e., one would need to exhibit $\hat{Z}$ that is a feasible solution to \eqref{eq:dual} and satisfies $\langle \matrixones, \hat{Z} \rangle = \ks$. To show that $(\ac,\ks)$ is the {\em unique} optimal solution to \eqref{eq:lovasz_lambdamax},  it is necessary to appeal to a strict complementarity-type of result for SDPs; see, e.g., \cite{overton,LV}.  In what follows, we specialize a set of conditions from \cite{LV} that guarantee unique solutions in SDPs to the Lov\'asz theta function.

\begin{theorem}\label{thm:constraintqualification}
Let $(t,A)$ and ${Z}$ be a pair of strict complementary primal and dual optimal solutions to \eqref{eq:lovasz_lambdamax} and \eqref{eq:dual} respectively, i.e., they satisfy:
$$\la X , Z \ra = 0 \quad \text{ and } \quad \mathrm{rank}(X) + \mathrm{rank}(Z) = |\V|, \qquad \text{ where } \qquad X = tI + A - \matrixones. $$  
Then $A$ is the unique primal optimal solution to \eqref{eq:lovasz_lambdamax} if and only if $A$ is an extreme point of the feasible region of \eqref{eq:lovasz_lambdamax}.
\end{theorem}
In view of Theorem \ref{thm:constraintqualification}, to show that $(\ac,\ks)$ is the unique optimal solution of \eqref{eq:lovasz_lambdamax}, it suffices to $(i)$ show that $(\ac,\ks)$ is an extreme point of the feasible region of the feasible region of \eqref{eq:lovasz_lambdamax}, and $(ii) $ produce a dual optimal $\hat{Z} \in \mathbb{R}^{|V| \times |V|}$ that satisfies strict complementarity.

The answer to the first task has simple answer: in Section \ref{sec:extremality}, we show that $(\ac,\ks)$ is an extreme point of the feasible region of \eqref{eq:lovasz_lambdamax} whenever the graph $G$ satisfies the $\cscc$ condition for some $c < 1$.  Our proof relies on results that characterize extreme points of spectrahedra, and the specific version we use is found in Corollary 3 of \cite{RG:95}:
\begin{theorem}\label{thm:extremal_rg}
Let $\mathcal{S}$ be a spectrahedron specified in linear matrix inequality (LMI) form
\begin{equation} \label{eq:LMI}
\mathcal{S}=\Big \{ (x_1, \ldots, x_n)^T \in \R^n : Q_0 +\sum_{i=1}^n x_i Q_i \succeq 0 \Big \},
\end{equation} 
where $Q_0,Q_1,\ldots, Q_n \in \mathbb{R}^{m \times m}$ are symmetric matrices.  Let $\by = (y_1,\ldots, y_n)^T \in \mathcal{S}$, and $U$ be an $m\times k$ matrix whose columns form a basis of the kernel of the matrix $Q_0 +\sum_i y_i Q_i$. Then, $\by$ is an extreme point of $\mathcal{S}$ if and only if the vectors ${\rm vec}(Q_1U), \ldots, {\rm vec}(Q_nU)$ are linearly independent. 
\end{theorem}

The answer to the second task is considerably more involved.  As a reminder, a matrix $Z \in~\mathbb{R}^{|V| \times |V|}$ is optimal  for \eqref{eq:dual}  if it satisfies:
\begin{align} 
& Z \succeq 0 \label{cond1}, \\    
& Z_{i,j} = 0 \text{ for all } (i,j) \in \E \label{eq:supportcondition}, \\
& \langle Z, \xc \rangle = 0 \label{eq:complementaryslackness}, \\ 
& \mathrm{tr}(Z) = 1. \label{eq:traceone}
\end{align}
We call a matrix $Z$ satisfying these conditions a {\em dual certificate}.  Furthermore, we say that a dual certificate  $Z$ and $\xc$ satisfy  {\em strict complementarity}~if:
\begin{equation}
{\rm rank}(Z)=|\V| - {\rm rank}(\xc) \label{cond3}.
\end{equation}

Note that we can omit condition \eqref{eq:traceone} as we can always scale a matrix that satisfies conditions \eqref{cond1}, \eqref{eq:supportcondition},  \eqref{eq:complementaryslackness} and  \eqref{cond3} to make it trace-one.  The proof our two main results   is given in three steps.

\paragraph{Step 1: Exact recovery in the case of disjoint cliques.}  In the first step, we consider the simplest setting where $G$ is the disjoint union of cliques $\{\ccs_l\}_{l=1}^{\ks}$.  With a bit of guesswork, we construct the following matrix $\zz$ defined as follows:
\begin{equation}\label{zmatrix}
\zz := \left(\begin{array}{ccc}
\frac{I}{|\ccs_1|} & \frac{1}{|\ccs_1||\ccs_2|} \matrixones & \ldots \\
\frac{1}{|\ccs_1||\ccs_2|} \matrixones & \frac{I}{|\ccs_2|} & \ldots \\
\vdots & \vdots & \ddots
 \end{array}\right).
\end{equation}
Here, the $(i,j)$-th block has dimension $|\ccs_i| \times |\ccs_j|$.

It is relatively straightforward to verify that the conditions \eqref{eq:supportcondition} and \eqref{eq:complementaryslackness} are satisfied.  In  Proposition \ref{thm:Unequalsizecanonicalmatrix_spectrum} we characterize the spectrum of the matrix $\zz$, and in so doing, show that $\zz$ satisfies the conditions \eqref{cond1} and \eqref{cond3}.  We explain these steps in Sections \ref{sec:extremality} and \ref{sec:disjointunion}.

\paragraph{Step 2: Exact recovery  for graphs with the $\cscc$ property.}  The matrix $\zz$ that served as our dual certificate for graphs formed by disjoint cliques does not work in the more general setting as it is non-zero on every edge between distinct cliques; i.e., it violates \eqref{eq:supportcondition}.  The key idea for constructing a suitable dual certificate in this case begins by recognizing that we can satisfy conditions \eqref{eq:supportcondition} and \eqref{eq:complementaryslackness} by specifying suitable subspaces.  Concretely, we define
$$
\begin{aligned}
\mathcal{K} & := \{ Z : Z_{i,j} = 0 \text{ for all } (i,j) \in \E \}, \quad \text{ and } \\
\mathcal{L} & := \{ Z : Z(\ks\bone_{\ccs_l} -\be)=0 \text{ for all } l \in [\ks] \}.
\end{aligned}
$$
At a first glance, it is not clear how the subspace $\mathcal{L}$ relates to condition \eqref{eq:complementaryslackness}.  In fact, $\mathcal{L}$ actually specifies a complementary slackness type of condition $\xc Z = 0$.  Later, in Proposition \ref{csexpnaded}, we show that the condition specified by $\mathcal{L}$ coincides with the condition $\langle Z,\xc \rangle = 0$ whenever $Z$ is PSD.  In view of the new subspaces we define, we consider a dual certificate $\zc$ obtained by projecting the matrix $\zz$ with respect to the Frobenius norm onto the intersection of $\mathcal{K}$ and $\mathcal{L}$:
$$
\zc ~~:=~~ \underset{X}{\arg \min} ~\| X - \zz \|_F^2 \quad \mathrm{s.t.} \quad X \in \mathcal{K}\cap \mathcal{L}.
$$
The remainder of proof is a matrix perturbation argument in which we show that $\|\zc - \zz \|$ is quite small, which allows us to conclude the remaining conditions \eqref{cond1} and \eqref{cond3}.
We now provide some geometric intuition why  the  projection onto the intersection $\mathcal{K}\cap \mathcal{L}$ succeeds.  

The  graphs considered in this work  are either disjoint union of cliques or have the $\cscc$ property.  In both cases, we have that $\mathcal{K}$ is a subspace of  
$$\mathcal{M} := \{ X \in \mathbb{R}^{|\V| \times |\V|}: \ X=X^T, (\ccs_l,\ccs_l) \text{-th block is diagonal for all } 1 \leq l \leq k \},
$$
which we treat as the ambient space, rather than the space of all symmetric matrices.  We set
$$\tilde{\mathcal{K}} := \mathcal{K} \cap \mathcal{M}  \ \text{ and } \  \tilde{\mathcal{L}}:= \mathcal{L} \cap \mathcal{M}.$$
Then $\zc$ can be equivalently defined as the projection of $\zz$ onto the intersection of $\tilde{\mathcal{K}} \cap \tilde{\mathcal{L}}$.
\begin{equation} \label{eq:euclideanprojection}
\zc ~~:=~~ \underset{X}{\arg \min} ~\| X - \zz \|_F^2 \quad \mathrm{s.t.} \quad X \in \tilde{\mathcal{K}} \cap \tilde{\mathcal{L}}.
\end{equation}
In Section \ref{sec:proof-deterministicrecovery} we give the main technical result of this work where we show that the subspaces $ \tilde{\mathcal{K}}^\perp$ and $ \tilde{\mathcal{L}}^\perp$ are ``almost orthogonal''.  The proof of this step heavily requires that the ambient space be specified as $\mathcal{M}$.  Subsequently, using consequences from first order optimality, we show that near orthogonality implies exact recovery.

\paragraph{Step 3: Recovery for planted clique cover instances.}  Our third  step shows that planted clique instances are $\cscc$, with high probability.  The proof is a direction application of Hoeffding's inequality combined with a union bound, and is provided in Section \ref{sec:randomgraphs}.

\section{Extremality of $(\ac,\ks)$} \label{sec:extremality}

The goal of this section is to provide a simple sufficient condition that guarantees when $(\ac,\ks)$ is an extreme point of the feasible region of \eqref{eq:lovasz_lambdamax}.  As we noted in Section \ref{sec:outline}, we rely  on a  general result that describes extreme points of spectrahedra specified by a linear matrix inequality (LMI), stated in the form of Theorem \ref{thm:extremal_rg}.  We begin by expressing the feasible region of \eqref{eq:lovasz_lambdamax} as the following LMI
$$
\Big \{  (t, a_{i,j})_{(i,j) \in \E} \in \mathbb{R}^{|\E|+1} : - \matrixones + t I + \sum_{(i,j) \in \E} a_{i,j} E_{i,j} \succeq 0 \Big \}.
$$
Here we use the notation $E_{i,j}={1\over 2}(\be_i \be_j^T+\be_j \be_i^T)$.  This suggests we should take the matrices $\{I\} \cup \{ E_{i,j} : (i,j) \in \E \}$ to be $\{Q_i\}_{i=1}^{n}$, and to identify the matrix $\xc$ with $Q_0 +\sum_i x_i Q_i$ in Theorem \ref{thm:extremal_rg}.  Our next task is to characterize the kernel of~$\xc$.

\subsection{Computing the kernel of $\xc$}

\begin{proposition}\label{thm:jspace} Setting 
$$
\mathcal{J}:=\big \{ \bx \in \R^{|\V|} : \la \bx, \bone_{\ccs_1} \ra = \ldots = \la \bx,  \bone_{\ccs_{\ks}}\ra \big\},
$$
 we have that
$$\mathcal{J} = \lspan\{\bone_{\ccs_1} - \bone_{\ccs_l}:  2\le l\le \ks \}^\perp = \lspan\{ \ks \bone_{\ccs_l} -\be: l\in [\ks]\}^\perp.$$
\end{proposition}

\begin{proof}  The first equality follows immediately from definition since $\la \bx, \bone_{\ccs_1} \ra = \la \bx, \bone_{\ccs_l} \ra \Leftrightarrow \la \bx, \bone_{\ccs_1} - \bone_{\ccs_l} \ra = 0$.  We focus on the second equality.  To do so, it suffices to show that every vector of the form $\ks \bone_{\ccs_l} -\be$, $l\in [\ks]$, is in the span of $\bone_{\ccs_1} - \bone_{\ccs_l}$, $2\le l\le \ks $, and vice versa.

First, we have $\bone_{\ccs_i} - \bone_{\ccs_j} = (\bone_{\ccs_1} - \bone_{\ccs_j}) - (\bone_{\ccs_1} - \bone_{\ccs_i})$.  Then, note that $\ks \bone_{\ccs_l} -\be = \sum_{i=1}^{\ks} (\bone_{\ccs_l} - \bone_{\ccs_i})$.  Hence the vectors $\ks \bone_{\ccs_l} -\be $ lies in the linear span of vectors of the form $\bone_{\ccs_1} - \bone_{\ccs_l}$, $2\le l\le \ks $.

In the reverse direction, note that $\bone_{\ccs_1} - \bone_{\ccs_l} = \frac{1}{\ks}(\ks \bone_{\ccs_1} -\be) - \frac{1}{\ks}(\ks \bone_{\ccs_l} -\be)$.  In other words, every vector of the form $\bone_{\ccs_1} - \bone_{\ccs_l}$, $2\le l\le \ks$, lies in the linear span of vectors of the form $\ks \bone_{\ccs_l} -\be$, $l\in [\ks]$.  This establishes the second equality.
\end{proof} 

As an immediate consequence we have that:

\begin{proposition}\label{kernellemma}
The kernel  of $\xc$ is equal to  $ \mathcal{J}.$ 
\end{proposition}

\begin{proof} By definition of $\xc$ we have that 
${\rm range}(\xc)={\rm span}\{ \ks \bone_{\mathcal{C}_l}-\be: l \in [k]\}.$
Thus,
$${\rm ker}(\xc) ={\rm span}\{ \ks \bone_{\mathcal{C}_l}-\be: l \in [k]\}^\perp.$$

\end{proof}

\begin{proposition}\label{csexpnaded}
The following  statements are equivalent for a PSD matrix $Z \in \mathbb{R}^{|\V| \times |\V|}$:

\begin{itemize}
\item[(1)]  $  Z\in  {\rm span}( \xc)^\perp$.
\item[(2)] $Z \in \mathcal{L}$; that is, $Z(\ks\bone_{\ccs_l} -\be)=0,$ for all $l \in [\ks]$.
\item[(3)] $\range(Z) \subseteq \mathcal{J}$.
\item[(4)] $\bz_i \in \mathcal{J}$ for all $i \in \V$ -- here, $\bz_i$ are rows of $Z$.
\end{itemize} 
Moreover, we have that ${\rm rank}(\xc)+{\rm rank}(Z)=|\V|$ if and only if   $\range(Z) = \mathcal{J}$.

\end{proposition}

\begin{proof}$(1)\Longleftrightarrow (2)$ We have that 
$$\la Z, \xc\ra=0 \iff \sum_l(\ks\bone_{\ccs_l}-\be)^T Z (\ks\bone_{\ccs_l}-\be)=0\iff Z(\ks\bone_{\ccs_l} -\be)=0, \  \forall l \in [\ks]$$ 
 where the last equivalence holds as  $Z\succeq 0$. 

$(2)\Longleftrightarrow (3)$ Note that (2) is equivalent to  
$$\ker(Z)\supseteq  \range(\xc)=\lspan\{ \ks \bone_{\ccs_l} -\be: l\in [\ks]\}.$$ 
Taking orthogonal complements, this is in turn equivalent to 
$$\range(Z)\subseteq \lspan\{ \ks\bone_{\ccs_l} -\be: l\in [\ks]\}^\perp=\mathcal{J}.$$
Finally $(3) \iff (4)$ since $\range(Z)=\range(Z^T)$. 
\end{proof}

Our next result describes the spectrum of the matrix $\zz$ 
and  show that the columns of the matrix $\zz$ span the kernel of $\xc$. Recall that
\begin{equation}\label{zfullmatrix}
\zz := \left(\begin{array}{ccc}
\frac{I}{|\ccs_1|} & \frac{1}{|\ccs_1||\ccs_2|} \matrixones & \ldots \\
\frac{1}{|\ccs_1||\ccs_2|} \matrixones & \frac{I}{|\ccs_2|} & \ldots \\
\vdots & \vdots & \ddots
 \end{array}\right),
\end{equation}
This matrix can be alternatively expressed in a more convenient form:
$$
\zz = \mathrm{diag}(\bg) + \bg\bg^T - \sum_{l} \frac{\bone_{\ccs_l}\bone_{\ccs_l}^T}{|\ccs_l|^2} \quad \text{where} \quad \bg = \Big(\underbrace{\frac{1}{|\ccs_1|}, \ldots }_{|\ccs_1|} , \ldots,  \underbrace{\ldots, \frac{1}{|\ccs_{\ks}|} }_{|\ccs_{\ks}|} \Big)^T.
$$
To be clear, this is precisely the same matrix \eqref{zmatrix} we use as our dual certificate whenever $G$ is a disjoint union of cliques.  However, we will not require any information about $\zz$ being a dual certificate at this juncture.  
\begin{proposition} \label{thm:Unequalsizecanonicalmatrix_spectrum}
The eigenvalues of $\zz$ are $\sum_{l=1}^{\ks} 1/|\ccs_l|$ (with multiplicity one), $0$ (with multiplicity $\ks-1$), and $1/|\ccs_l|$ with multiplicity $|\ccs_l|-1$. Moreover, we have that 
$\range(\zz)=\mathcal{J}$.
\end{proposition}

\begin{proof}
First, note that the $\bg$ is an eigenvector whose eigenvalue is $\sum_{i=1}^{\ks} 1/|\ccs_i|$.  Second, note that the vector $\bone_{\ccs_1} - \bone_{\ccs_l} $ is an eigenvector with eigenvalue $0$, for all $2 \leq l \leq \ks$.  Third, fix a subset $\ccs_l$.  Let $\bx \in \mathbb{R}^{|V|}$ be a vector with entries in the coordinates corresponding to $\ccs_l$ satisfying $\be^T \bx = 0$.  One can check that $(\bg\bg^T - \sum_{l} (\bone_{\ccs_l}\bone_{\ccs_l}^T)/|\ccs_l|^2) \bx = 0$, and hence $\zz\bx = \mathrm{diag}(\bg) \bx = \bx / |\ccs_l|$.  The dimension of this eigenspace is $|\ccs_l|-1$.  Finally, we have that 
$$\ker(\zz)=\lspan\{\bone_{\ccs_1} - \bone_{\ccs_l}:  2\le l\le \ks \}$$
and using Proposition \ref{thm:jspace} we get that 
$$\range(\zz)=\lspan\{\bone_{\ccs_1} - \bone_{\ccs_l}:  2\le l\le \ks \}^\perp=\mathcal{J}.$$

\end{proof}

\subsection{Establishing extremality}

\begin{lemma} \label{thm:non_degen_Z} Consider  a graph $G$ with a  clique cover $\{\ccs_l\}_{l=1}^{\ks}$.  Let $\mathcal{S} \subset \V$ be a subset of vertices such that, for all cliques except one, $ \mathcal{S}$ leaves out at least one vertex.  Then, the columns of $\zz$ corresponding to $\mathcal{S}$ are linearly independent.
\end{lemma}

\begin{proof} Without loss of generality, we assume that for all cliques except $ \ccs_1$, the set of vertices $\mathcal{S}$ contains at most $ |\ccs_{l}|-1$ vertices from clique $\ccs_{l}$, for all $l\ne 1$.

We index the columns of $\zz$ by $(l,j)$, where the index $1\leq l \leq \ks$ refers to the clique while the index $j\in \ccs_l$ refers to the vertex within the clique.  Let $\{\bv_{l,j}\}$ be the columns of the matrix $\zz$ and consider a linear~combination:
\begin{equation}\label{eq:non_degen_Z_1}
\sum_{(l,j) \in \mathcal{S}} \theta_{l,j} \bv_{l,j} = 0.
\end{equation}
Fix a cluster $\ccs_l$, and let $P_l$ be the projection of $\mathbb{R}^{|\V|}$ onto the coordinates corresponding to $\ccs_l$.  By applying $P_l$ to \eqref{eq:non_degen_Z_1} we get:
\begin{equation}\label{eq:non_degen_Z_2}
\sum_{j : (l,j)\in \mathcal{S}} \theta_{l,j} P_l(\bv_{l,j}) = - \sum_{(l',j) \in \mathcal{S}, 1\leq l' \neq l \leq \ks} \theta_{l',j} P_l(\bv_{l',j}) =  c (1,\ldots,1)^T.
\end{equation}
The rightmost equality follows by noting that $P_l(\bv_{l',j}) = \be \in \mathbb{R}^{|\ccs_l|}$ for all $j$.  On the other hand, $P_l(\bv_{l,j}) = \be_j  \in \mathbb{R}^{|\ccs_l|}$ are standard basis vectors.  Since $\mathcal{S}$ leaves out at least one vector in $\ccs_l$, the span of these vectors cannot include $\be$.  So it follows from \eqref{eq:non_degen_Z_1} that $c=0$, and in particular, $\theta_{l,j} = 0$ for all $j$.  We repeat this argument for all $2 \leq l \leq \ks$ and conclude that $\theta_{l,j} = 0$ for all $l \geq 2$, and all $j$.

Finally, we consider the vectors $\bv_{1,j}$ and project these to the rows in $\ccs_1$.  The vectors $\{P_l(\bv_{1,j})\}$ are standard basis vectors.  Since $\sum_{(1,j) \in \mathcal{S}} \theta_{1,j} P_l(\bv_{1,j}) = 0$, it must be that $\theta_{1,j} = 0$ for all $j$.  Hence all the coefficients are zero.  This proves that the columns in $\mathcal{S}$ are linearly independent.
\end{proof}

\begin{lemma} \label{thm:non_degen_alg} Let $G = (\V,\E)$ be a graph that satisfies the $\cscc$ property for some $c<1$.  Then the collection of matrices 
$
\{ E_{i,j} \zz : (i,j) \in \E \} \cup \{ \zz \}
$ 
are linearly independent.
\end{lemma}

\begin{proof}
Consider a linear combination  
$$
\sum_{(i,j) \in \E} \theta_{i,j} E_{i,j} \zz  + \theta \zz = 0.
$$
Fix a vertex $m \in \ccs_l$, and for all $l' \neq l$, we let $N_{l'}(m) \subset \ccs_{l'}$ denote the neighbors of vertex $m$ in $\ccs_{l'}$.  Then the $m$-th row of the above expression is given by:
$$
\sum_{\substack{ j\in N_{l'}(m), l'\ne m}} \theta_{m,j} \bz^\star_j + \theta \bz^\star_m=0,
$$
where $\bz_j^\star$ denotes the $j$-th row of $\zz$.  Now, since the vertex $m$ is not fully connected to clusters $\ccs_{l'}$ for all $l' \neq l$, the collection of vertices $\{N_{l'}(m)\}_{l' \neq m}\cup \{m\}$ satisfy the conditions of the subset $\mathcal{S}$ in Lemma \ref{thm:non_degen_Z}.  By Lemma \ref{thm:non_degen_Z}, the corresponding columns of $\zz$ are linearly independent, which means that $\theta = \theta_{m,j} = 0$ for all $j$.  The result follows by repeating this argument for all $m$.
\end{proof}

\begin{theorem} \label{thm:extremepoint}
Let $G = (\V, \E)$ be a graph, and let $\{\ccs_l\}_{l=1}^{k}$ be a clique cover for $G$.  Suppose $G$ satisfies the $\cscc$ property for some $c<1$.  Then $(\ac,\ks)$ is an extreme point of the feasible region of~\eqref{eq:lovasz_lambdamax}.
\end{theorem}

\begin{proof}[Proof of Theorem \ref{thm:extremepoint}]  Suppose $G$ satisfies the $\cscc$ property for some $c<1$.  By Lemma \ref{thm:non_degen_alg}, the collection of matrices $\{ E_{i,j} \zz : (i,j) \in \E \} \cup \{ \zz \}$ are linearly independent.  Hence the vectors in $\{ {\rm vec}( E_{i,j} \zz) : (i,j) \in \E \} \cup \{ {\rm vec}(\zz) \}$ are also linearly independent.  This means that the  following matrix formed by combining all the (vectorized) matrix product $E_{i,j} \zz$, ranging over all edges $(i,j) \in \E$, as well as $\zz$, has full column rank
$$
\left( \begin{array}{c|c|c|c}
\ldots & \mathrm{vec}( E_{i,j} \zz) & \ldots
& \mathrm{vec}(\zz)\end{array} \right)_{(i,j) \in \E}.
$$ 
Hence, by Theorem \ref{thm:constraintqualification}, $(\ac,\ks)$ is an extreme point of the feasible region of \eqref{eq:lovasz_lambdamax}.
\end{proof}

\section{Exact recovery in the case of  disjoint cliques} \label{sec:disjointunion}

In this section, we show that $(\ac,\ks)$ is the unique optimal solution to the Lov\'asz theta formulation~\eqref{eq:lovasz_lambdamax} when $G$ is a disjoint union of cliques.  In this simple case, $G$ satisfies the $\cscc$ condition with $c=0$.  Following the conclusions of Theorem \ref{thm:extremepoint}, $(\ac,\ks)$ is an extreme point of the feasible region of \eqref{eq:lovasz_lambdamax}.  Based on Theorem \ref{thm:constraintqualification}, it remains to produce a suitable dual witness $Z$ satisfying the requirements listed in Theorem \ref{thm:constraintqualification}.

As it turns out, the matrix $\zz$ satisfies all the requirements we need!  In fact, in the process of computing the spectrum of the matrix $\zz$ in Proposition \ref{thm:Unequalsizecanonicalmatrix_spectrum}, we have also verified that $\zz$ does indeed satisfy the requirements of Theorem \ref{thm:constraintqualification}.  We collect these conclusions below.
 
\begin{theorem} \label{thm:disjointcliques_uniquerecovery}
Let $G$ be the graph formed by the disjoint union of cliques $\{\ccs_l\}_{l=1}^{\ks}$.  Then the matrix $(\ac,\ks)$ is the unique solution of \eqref{eq:lovasz_lambdamax}.
\end{theorem}

\begin{proof}
As we noted above, we simply need to check that the matrix $(1/\ks)\zz$ satisfies conditions \eqref{cond1} to \eqref{cond3}.  First, from Proposition \ref{thm:Unequalsizecanonicalmatrix_spectrum}, the matrix $\zz$ is PSD, which is \eqref{cond1}.  Second, all the edges of $G$ are within cliques.  The matrix $\zz$, when restricted to each clique $\ccs_l$, is the identity matrix, and thus satisfies the support condition \eqref{eq:supportcondition}.  It is easy to verify that all the rows of $\zz$ belong to~$\mathcal{J}$.  Hence, by Proposition \ref{csexpnaded}, $\zz\in {\rm span}(\xc)^\perp$, which is \eqref{eq:complementaryslackness}.  It is easy to see that $\mathrm{tr}(\zz) = \ks$ and $\langle \matrixones, \zz \rangle = (\ks)^2$, which shows  \eqref{eq:traceone}.  Last, by Proposition \ref{thm:Unequalsizecanonicalmatrix_spectrum}, we have $\range(\zz)=\mathcal{J}$.  Hence, by Proposition \ref{csexpnaded}, we have ${\rm rank}(\xc)+{\rm rank}(\zz)=|\V|$, which is~\eqref{cond3}.
\end{proof}

\section{Exact recovery for graphs with the $\cscc$ proeprty} \label{sec:proof-deterministicrecovery}

In this section, let $G$ be a graph that satisfies the $\cscc$ property.  Our goal is to extend the results in Section \ref{sec:disjointunion}, and show that $(\ac,\ks)$ is also the unique solution to \eqref{eq:lovasz_lambdamax} for graphs satisfying the $\cscc$ condition, provided $c$ is appropriately small.  Following our discussion in Section \ref{sec:disjointunion}, the matrix $\xc$ remains an extreme point of the feasible region of \eqref{eq:lovasz_lambdamax} if $c<1$.  Hence, by Theorem \ref{thm:constraintqualification}, it remains to produce a suitable dual witness $Z$ satisfying the requirements listed in Theorem \ref{thm:constraintqualification}.

\subsection{Incoherence-type result} \label{sec:incoherence}

The certificate we will use  for graphs with the  $\cscc$ property is   $\zc = P_{\tilde{\mathcal{K}} \cap \tilde{\mathcal{L}}} (\zz)$, defined as  the projection of $\zz$ onto $  {\tilde{\mathcal{K}} \cap \tilde{\mathcal{L}}}$ with respect to the Frobenius norm, i.e., 
$$\zc ~~:=~~ \underset{X}{\arg \min} ~\| X - \zz \|_F^2 \quad \mathrm{s.t.} \quad X \in \tilde{\mathcal{K}} \cap \tilde{\mathcal{L}}.$$
Using the first-order optimality conditions we have that 
\begin{equation} 
\zz - \zc = \tilde{L}  + \tilde{K},
\end{equation}
where $\tilde{L}$ lies in the normal cone of $\mathcal{L}$ at $Z'$ and $\tilde{K}$ lies in the normal cone of $\mathcal{K}$ at $Z'$.  Recalling that the normal cone of a subspace is its orthogonal complement we have that $\tilde{L}\in \tilde{\mathcal{L}}^\perp$ and $\tilde{K}\in \tilde{\mathcal{K}}^\perp$.  As a reminder, these orthogonal complements are defined with respect to $\mathcal{M}$ being the ambient space.

Our main result  can be viewed as the extension of orthogonality for graphs that are disjoint unions of cliques to graphs with the $\cscc$ property.  The proof is broken up into a sequence of results that follow.

\begin{theorem} 
 \label{thm:incoherence_1}
Let $G$ be a graph satisfying the $\cscc$ property.  Then, 
\begin{itemize}
\item[$(i)$] For any  $\tilde{K} \in \tilde{\mathcal{K}}^\perp$ and  $\tilde{L} \in \tilde{\mathcal{L}}^\perp$ we have   
$$
| \langle \tilde{K}, \tilde{L} \rangle | \leq 2\sqrt{c}\|\tilde{K}\|_F  \|\tilde{L}\|_F .
$$
\item[$(ii)$]  For any  $\tilde{L} \in \tilde{\mathcal{L}}^\perp$ we have 
$$\|P_{\mathcal{K}^\perp}(\tilde{L})\|_F \leq (2 \sqrt{c})^{1/2} \|\tilde{L}\|_F.$$
\end{itemize}
\end{theorem}

\paragraph{A direct sum decomposition for  $\tilde{\mathcal{L}}^\perp$.}  The first  step is to provide a decomposition of the space $\tilde{\mathcal{L}}^\perp$ into simpler subspaces, on which it is easier to prove the near orthogonality property.  We use these results as basic ingredients to build up to our near orthogonality property later.

Define the matrices $\{ F_{x,y,z} : 1 \leq x \neq y \leq k, 1\leq z \leq |\ccs_x| \}$ so that (i) the $(\ccs_x,\ccs_x)$-th block is a diagonal matrix whose $z$-th entry is set equal to $-2(|\ccs_x|-1)$ and all remaining entries equal to $2$, and (ii) the $(\ccs_x,\ccs_y)$-th block ($(\ccs_y,\ccs_x)$-th block) is such that entries in the $z$-th row (column) are equal to $|\ccs_x|-1$ and all other entries equal to $-1$, (iii) all other entries are zero. As an example, in the case where  $k=2$ the matrix  $F_{1,2,1} $ is given by:
\begin{equation} \label{eq:f121}
F_{1,2,1} := \left( \begin{array}{cccc|ccc}
-2(|\ccs_1|-1) &&&& |\ccs_1|-1 & \ldots & |\ccs_1|-1  \\
& 2 &&& -1 & \ldots & -1  \\
&& \ddots &&\vdots &&\vdots  \\
&&&2& -1 & \ldots & -1  \\
\hline
|\ccs_1|-1 & -1 & \ldots & -1 &&& \\
\vdots & \vdots & & \vdots &&& \\
|\ccs_1|-1 & -1 & \ldots & -1 &&& \\
\end{array} \right),
\end{equation}
where omitted entries are zero.  Second, consider the following subspaces:\medskip

\begin{table}[h]
\centering
\begin{tabular}{|c|c|c|}
\hline
Name & Description & Dimension \\
\hline \hline
$\mathcal{T}_{1}$ & $ \left \{ \left( \begin{array}{c|c|c|c|c} \gamma_1 I & \theta_{1,2} \matrixones & \theta_{1,3} \matrixones & \ldots & \theta_{1,n} \matrixones \\ \hline \theta_{2,1} \matrixones & \gamma_2 I & \theta_{2,3} \matrixones & \ldots & \theta_{2,n} \matrixones \\ \hline \theta_{3,1} \matrixones & \theta_{3,2} \matrixones & & & \\ \hline \vdots & \vdots & & & \\ \hline \theta_{n,1} \matrixones & \theta_{n,2} \matrixones & & & \gamma_n I \end{array}\right) : \sum \gamma_i + \sum \theta_{i,j} = 0 \right\}$ & ${k+1 \choose 2}-1$ \\ \hline
$\mathcal{T}_{2} $ & $ \mathrm{Span} \left\{ F_{x,y,z} : 1 \leq x \neq y \leq k, 1\leq z \leq |\ccs_x| \right\}$ & $(|V| - k) \times (k-1)$ \\ \hline
\end{tabular}
\end{table}

\begin{proposition} \label{thm:description_of_lperp} We have that 
$$\tilde{\mathcal{L}}^\perp=\mathcal{T}_{1}\oplus \mathcal{T}_{2}.$$
 
\end{proposition}

\begin{proof}[Proof of Proposition \ref{thm:description_of_lperp}]
The fact that   $\mathcal{T}_{1}$ and $\mathcal{T}_{2}$ are  orthogonal is easy to check. Define the matrix $L_{x,y,z} \in \mathbb{R}^{|\V|\times|\V|}$, where $1 \leq x \neq y \leq k$, and $1 \leq z \leq |\ccs_x|$ such that (i) the $z$-th entry of the $(\ccs_x,\ccs_x)$-th block is equal to $2$, and (ii) the entries of the $z$-th row (column) of the $(\ccs_x,\ccs_y)$-th block ($(\ccs_y,\ccs_x)$-th block) are all equal to $-1$, and (iii) all other entries are $0$.

As an example, in the case of two cliques (so $k=2$)  the  matrix $L_{1,2,1} $ is given by:
$$
L_{1,2,1} := \left( \begin{array}{cc|ccc} 
2 & & -1 & \ldots & -1  \\
& & & &  \\
\hline
-1 & & & &   \\
\vdots & & & &  \\
-1 & & & &   \\
\end{array}\right),
$$
where all omitted entries are zero. 

First, observe that the matrices $\{ L_{x,y,z} \}$ specify all the linear equalities in the subspace $\tilde{\mathcal{L}}$, and thus, $\tilde{\mathcal{L}}^\perp$ is precisely the span of $\{ L_{x,y,z} \}$.  As such, to show that   $\mathcal{T}_{1}$ and $\mathcal{T}_{2}$  span $\tilde{\mathcal{L}}^\perp$, it suffices to show that every matrix $L_{x,y,z}$ is expressible as the sum of matrices belonging to each of these subspaces.  
As a concrete example, we show this is true for   $L_{1,2,1}$ -- the construction for other matrices are similar.
Indeed, one checks that $L_{1,2,1}$ is the linear sum of these matrices
$$ 
L_{1,2,1}= \underbrace{\frac{1}{|\ccs_1|} \left( \begin{array}{c|c|c|c} 2I & -\matrixones & 0 & \ldots \\ \hline -\matrixones & 0 & \ldots & \\ \hline 0 & \vdots & \ddots & \\ \hline \vdots & & & \end{array}\right) }_{\mathcal{T}_1}
 -\frac{1}{|\ccs_1|} \underbrace{F_{1,2,1}}_{\mathcal{T}_2}.$$
This completes the proof.
\end{proof}

In view of Proposition \ref{thm:description_of_lperp}, to bound the inner product between vectors  in $\tilde{\mathcal{K}}^\perp$ and $\tilde{\mathcal{L}}^\perp$ respectively, it suffices to bound inner product between $ (i) \ \tilde{\mathcal{K}}^\perp$  and $\mathcal{T}_1$ and $ (ii) \ \tilde{\mathcal{K}}^\perp$  and $\mathcal{T}_2$.

Next, we describe a result that shows how incoherence computation for (a small number of) orthogonal subspaces can be put together to obtain incoherence computations.  

\begin{lemma}\label{thm:orthogonalsubspace_incoherence}  Let  $\{ \mathcal{T}_i \}_{i=1}^{r}$ be orthogonal subspaces of $\R^d$. Consider $s\in \R^d$ such that 
$$
| \langle s,t_i \rangle | \leq \epsilon \| s \|_2 \| t_i \|_2 \quad \text{ for all }  t_i \in \mathcal{T}_i.
$$
Then, for $t\in \oplus_i \mathcal{T}_i $ we have that 
$$
| \langle s,t \rangle | \leq \epsilon \sqrt{r}  \| s \|_2 \| t \|_2.
$$
\end{lemma}

\begin{proof}[Proof of Lemma \ref{thm:orthogonalsubspace_incoherence}]
For any  $t= \sum t_i\in \oplus \mathcal{T}_i$ we have 
$$| \langle s, t \rangle | = | \langle s, \sum_{i=1}^r t_i \rangle | \leq \sum_{i=1}^r | \langle s, t_i \rangle | \leq \epsilon \sum_{i=1}^r \| s \|_2 \| t_i \|_2 \leq \epsilon \sqrt{r} \|s\|_2(\sum_{i=1}^r \| t_i \|_2^2)^{1/2} = \epsilon\sqrt{r}\|s\|_2  \| t \|_2,$$
where   the second last inequality follows from Cauchy-Schwarz, while the last equality uses the fact that $\mathcal{T} = \oplus \mathcal{T}_i$.
\end{proof}

\paragraph{Bounding inner product between $\tilde{\mathcal{K}}^\perp$ and $\mathcal{T}_2$.}  Define the column vectors $\{\bf_i\}_{i=1}^{n}$ by
$$
\bf_i = n \be_i - \be = (\ldots, \underbrace{-1}_{i-1},\underbrace{n-1}_{i}, \underbrace{-1}_{i+1},\ldots)^T \in \mathbb{R}^{n}.
$$
Note that 
$$ \bf_i \be^T=\begin{pmatrix} \bf_i & \ldots & \bf_i\end{pmatrix} \text{ and } \be \bf_i^T=\begin{pmatrix} \bf_i^T \\ \vdots \\ \bf_i^T\end{pmatrix}.$$
With a slight abuse of notation, we define the subspace of matrices in $\mathbb{R}^{n_1 \times n_2}$ 
$$
\mathcal{N}_{C} ={\rm span}( \bf_1 \be^T,\ldots,  \bf_{n_1} \be^T ) \qquad \text{ and }  \qquad \mathcal{N}_{R}={\rm span}( \be \bf_1^T, \ldots, \be \bf_{n_2}^T ).
$$
The abuse of notation arises because the vectors $\bf_i$ in the definitions of $\mathcal{N}_{C}$ and $\mathcal{N}_{R}$ are different.  Note that because $\bf_i^T \be = 0$, the subspaces $\mathcal{N}_{C}$ and $\mathcal{N}_{R}$ have dimensions $n_1-1$ and $n_2-1$ respectively and are orthogonal.  $\mathcal{N}_{C}$ and $\mathcal{N}_{R}$ are relevant for our problem as   the block off-diagonal entries of any matrix in  $\mathcal{T}_{2}$ belong to $ \mathcal{N}_{C} \oplus \mathcal{N}_{R} = \{ X_C + X_R : X_C \in \mathcal{N}_{C}, X_C \in \mathcal{N}_{R} \} $.

\begin{lemma} \label{thm:ncnr_space_incoherence}
Let $K \in \mathbb{R}^{n_1\times n_2}$  such that each row  has at most $c n_2$ non-zero entries for some  $0 \leq c \leq 1$. Then, for any $L \in \mathcal{N}_{C}$   we have that 
$$| \langle K, L \rangle | \leq \sqrt{c} \| L \|_F \| K \|_F.$$
 The same conclusion holds  if each column  of $K$ has at most $c n_1$ non-zero entries and   $L\in \mathcal{N}_R$. 
\end{lemma}

\begin{proof} [Proof of Lemma \ref{thm:ncnr_space_incoherence}]
Since $L \in \mathcal{N}_{C}$, all of its columns equal one another.  Suppose its column entries are $(\theta_1,\ldots,\theta_{n_1})$, i.e., $L_{x,y}=\theta_x$  for all $y$. We then have 
$$\langle K, L \rangle = \mathrm{tr}(KL^T) = \sum_{x} (\sum_{y} K_{x,y} L_{x,y}) = \sum_{x} \theta_x \sum_{y} K_{x,y} {\leq} \sum_{x} \theta_x \sqrt{cn_2} (\sum_{y} K_{x,y}^2)^{1/2},$$  where the last inequality follows from  Cauchy-Schwarz, and since   there are at most $c n_2$ non-zero entries in each column of $K$.  Finally, we  have 
$$\sqrt{cn_2} \sum_{x} \theta_x  (\sum_{y} K_{x,y}^2)^{1/2} \leq \sqrt{cn_2} (\sum_x \theta_x^2)^{1/2} (\sum_{x,y} K_{x,y}^2)^{1/2} = \sqrt{c} \|L\|_F \|K\|_F,$$
 where in the last equality we use  that $\|L\|_F = \sqrt{n_2}(\sum \theta_i^2)^{1/2}$. 
\end{proof}

\begin{corollary} \label{thm:ncnr_combined}
Consider   $K \in \mathbb{R}^{n_1\times n_2}$ where  each column has at most $c n_1$ non-zero entries, and each row has at most $c n_2$ non-zero entries for some  $0 \leq c \leq 1$.  For any  $L \in \mathcal{N}_{C}\oplus \mathcal{N}_{R}$  we have: 
$$| \langle K, L \rangle | \leq \sqrt{2c} \| L \|_F \| K \|_F.$$
\end{corollary}

\begin{proof} [Proof of Corollary \ref{thm:ncnr_combined}]
 By Lemma \ref{thm:ncnr_space_incoherence} for  $ L \in \mathcal{N}_C,$ or $ L\in  \mathcal{N}_R$ we have that
$$| \langle K, L \rangle | \leq \sqrt{c} \| L \|_F \| K \|_F.$$
As $\be^T \bf_j=0$,  $\mathcal{N}_{C}$ and $\mathcal{N}_{R}$ are orthogonal subspaces. 
The result follows by Lemma \ref{thm:orthogonalsubspace_incoherence}.
\end{proof}

\begin{lemma} \label{thm:incoherence_efsum}
Suppose $G$ is a graph satisfying the $\cscc$ property.  
Then, for any  $K \in \tilde{\mathcal{K}}^\perp$ and $L\in \mathcal{T}_2$ we  have that
$$| \langle K,L \rangle| \leq \sqrt{2c} \|K\|_F \|L\|_F.$$
\end{lemma}

\begin{proof} As $K \in \tilde{\mathcal{K}}^\perp$, its diagonal blocks are zero, and  also, entries corresponding to non-edges of $G$ are zero.  Moreover, as $G$ has the   $c$-SCC property, each  row (and column) of $K$  has at most $cn$ non-zero entries.
Let the block matrices be indexed by $(x,y)$, and let $L_{xy}$ and $K_{xy}$ denote the $xy$-th block.  Recall that  the block off-diagonal entries of any matrix in  $\mathcal{T}_{2}$ belong to 
$ \mathcal{N}_{C} \oplus \mathcal{N}_{R}.$ By Corollary~\ref{thm:ncnr_combined} we have $| \langle K_{xy}, L_{xy} \rangle | \leq \sqrt{2c} \| L_{xy} \|_F \| K_{xy} \|_F$.  By summing over the blocks and by applying Cauchy-Schwarz, we have 
$$
 \begin{aligned}
  | \langle K , L \rangle | &  \leq \sum_{x,y} |\langle K_{xy} , L_{xy} \rangle |= \sum_{x\ne  y} |\langle K_{xy} , L_{xy} \rangle | \\
& \leq \sqrt{2c} \sum_{x\ne y} \| K_{xy} \|_F \| L_{xy} \|_F \leq \sqrt{2c} (\sum_{x,y} \| K_{xy} \|_F^2)^{1/2} (\sum_{x,y} \| L_{xy} \|_F^2)^{1/2} = \sqrt{2c} \| K \|_F \| L \|_F.
\end{aligned} $$
\end{proof}

\paragraph{Bounding  the inner product between $\tilde{\mathcal{K}}^\perp$ and $\mathcal{T}_1$.}  This case is easier and the required result is given in the next lemma. 
\begin{lemma}\label{thm:incoherence_Espan}
Suppose $G$ is a graph satisfying the $\cscc$ property.  Then, for any  $K \in \tilde{\mathcal{K}}^\perp$ and $L\in \mathcal{T}_1$ we  have that
$$|\langle K, L \rangle | \leq \sqrt{c} \| K \|_F \| L \|_F .$$
\end{lemma}

\begin{proof}[Proof of Lemma \ref{thm:incoherence_Espan}]
Note that $K$ has no entries in each block diagonal.  Consider the $(i,j)$-th block matrix where $i \neq j$.  We denote the coordinates in this block by $\mathcal{B}_{i,j}$.  Then
$$
\sum_{x,y \in \mathcal{B}_{i,j}} K_{x,y} L_{x,y} = \theta_{i,j} \left( \sum_{x,y \in \mathcal{B}_{i,j}} K_{x,y} \right) \leq \sqrt{c |\ccs_i| |\ccs_j|} \theta_{i,j} \left( \sum_{x,y \in \mathcal{B}_{i,j}} K_{x,y}^2 \right)^{1/2}.
$$
The last inequality follows from Cauchy-Schwarz, and by noting that $K$ has at most $c |\ccs_i| |\ccs_j|$ non-zero entries within the block $\mathcal{B}_{i,j}$.  Then by summing over the blocks $(i,j)$ we have
$$
|\langle K,L \rangle| \leq \sum_{i,j} \sqrt{c|\ccs_i| |\ccs_j|} \theta_{i,j} ( \sum_{x,y \in \mathcal{B}_{i,j}} K_{x,y}^2 )^{1/2} \leq \sqrt{c} (\sum_{i,j} |\ccs_i| |\ccs_j|\theta_{i,j}^2)^{1/2} ( \sum_{x,y} K_{x,y}^2 )^{1/2}.
$$
The last inequality follows from Cauchy-Schwarz.  Now note that $( \sum_{x,y} K_{x,y}^2 )^{1/2} = \|K\|_F$, and that $(\sum_{i,j} |\ccs_i| |\ccs_j| \theta_{i,j}^2)^{1/2} = \|L\|_F $, from which the result follows.
\end{proof}

Finally, we are ready to prove  Theorem \ref{thm:incoherence_1}.

\paragraph{Proof of Theorem \ref{thm:incoherence_1}.} Consider   $\tilde{K} \in \tilde{\mathcal{K}}^\perp$ and  $\tilde{L} \in \tilde{\mathcal{L}}^\perp$.   By Proposition \ref{thm:description_of_lperp} we have that 
$\tilde{\mathcal{L}}^\perp=\mathcal{T}_1\oplus \mathcal{T}_2$ and let   $\tilde{L} =\tilde{L}_1+\tilde{L}_2$, where $\tilde{L}_i\in \mathcal{T}_i$.

\medskip 
    \textbf{Part $(i)$.} 
By Lemma \ref{thm:incoherence_Espan} we have that  
$$| \langle \tilde{K}, \tilde{L}_1 \rangle | \leq \sqrt{c}\|\tilde{K}\|_F  \|\tilde{L}_1 \|_F.$$  Write $\tilde{L}_2 = \tilde{L}_{2,\mathrm{diag}} + \tilde{L}_{2,\mathrm{off}}$ where $\tilde{L}_{2,\mathrm{diag}}$ is the block diagonal component of $\tilde{L}_2$ and where $\tilde{L}_{2,\mathrm{off}}$ is the block off-diagonal component $\tilde{L}_2$.  By the definition of $\tilde{L}_{2,\mathrm{off}}$, the block diagonal components of $\tilde{L}_{2,\mathrm{off}}$ are zero.  Next, we note that the off-diagonal block entries of the matrices $F_{x,y,z}$ belong to $\mathcal{N}_{C} \oplus \mathcal{N}_{R}$ -- see, for instance, $F_{1,2,1}$ in \eqref{eq:f121}. 
Hence by Lemma \ref{thm:incoherence_efsum}, we have 
$$| \langle \tilde{K}, \tilde{L}_{2,\mathrm{off}} \rangle | \leq \sqrt{2c} \| \tilde{K} \|_F \| \tilde{L}_{2,\mathrm{off}} \|_F.$$
Moreover, as the diagonal blocks of $\tilde{K}$ are zero we have that 
$$\langle \tilde{K}, \tilde{L}_{2,\mathrm{diag}} \rangle =0.$$
Putting everything together,
$$| \langle \tilde{K}, \tilde{L}_2 \rangle | = | \langle K, \tilde{L}_{2,\mathrm{off}} \rangle | \leq \sqrt{2c}  \| \tilde{K} \|_F \| \tilde{L}_{\mathrm{2,off}} \|_F \leq \sqrt{2c} \| \tilde{K} \|_F \| \tilde{L} \|_F,$$   
where the last inequality follows by noting that $\| L \|_F^2 = \| L_{\mathrm{off}} \|_F^2 + \| L_{\mathrm{diag}} \|_F^2$.

Finally, by Proposition \ref{thm:description_of_lperp}, the subspaces $\mathcal{T}_1$ and $\mathcal{T}_2$ are orthogonal.  Hence by Lemma \ref{thm:orthogonalsubspace_incoherence} applied to the subspaces $\mathcal{T}_1$ and $\mathcal{T}_2$ we have $| \langle \tilde{K}, \tilde{L} \rangle | \leq \sqrt{2} \times \sqrt{2c} \|\tilde{K}\|_F \|\tilde{L}\|_F$.

\medskip 

 \textbf{Part $(ii)$.} Note that
$$
\| P_{\mathcal{K}^\perp}(\tilde{L}) \|_F^2 = \langle P_{\mathcal{K}^\perp}(\tilde{L}), P_{\mathcal{K}^\perp}(\tilde{L}) \rangle = \langle P_{\mathcal{K}^\perp} (P_{\mathcal{K}^\perp}(\tilde{L})), \tilde{L} \rangle = \langle P_{\mathcal{K}^\perp}(\tilde{L}), \tilde{L} \rangle,
$$
where the second and third equalities  follow  from the fact that orthogonal projections are  self-adjoint and idempotent.  Since $P_{\mathcal{K}^\perp}(\tilde{L}) \in \mathcal{K}^\perp$, by Proposition \ref{thm:incoherence_1}, we have 
$$|\langle P_{\mathcal{K}^\perp}(\tilde{L}), \tilde{L} \rangle| \leq 2\sqrt{c} \|P_{\mathcal{K}^\perp}(\tilde{L})\|_F \|\tilde{L}\|_F \leq 2\sqrt{c} \|\tilde{L}\|_F^2,$$ from which the result follows.\qed

\subsection{Proof of Result \ref{thm:result2}} 

\begin{theorem} \label{thm:abcd-exact-recovery}
Suppose $G$ satisfies the $\cscc$ property for some $c < \min \{ \frac{1}{4} (\frac{\min_l 1/|\ccs_l|}{\sum_l 1/|\ccs_l|})^2,\frac{1}{100} \}$.  Then $(\ac,\ks)$ is the unique optimal solution to \eqref{eq:lovasz_lambdamax}.
\end{theorem}

Following Proposition \ref{thm:Unequalsizecanonicalmatrix_spectrum}, the largest singular value of $\zz$ is $\sigma_{\max}(\zz) = \sum_l 1/|\ccs_l|$, while the smallest non-zero singular value of $\zz$ is 
$$\sigma_{\ks}(\zz) = \min_l 1/|\ccs_l|.$$
 Hence an alternative way to express the lower bound in Theorem \ref{thm:abcd-exact-recovery} in terms of a condition number-type of quantity 
$$\frac{ \min_l 1/|\ccs_l| }{ \sum_l 1/|\ccs_l| }= \frac{ \sigma_{\ks}(\zz) }{ \sigma_{\max}(\zz) }.$$

\begin{proof}
As a reminder, our candidate dual certificate is $\zc = P_{\tilde{\mathcal{K}} \cap\tilde{\mathcal{L}}} (\zz)$ as defined earlier in \eqref{eq:euclideanprojection}.  By Theorem \ref{thm:extremepoint}, $\ac$ is an extreme point of the feasible region.  Thus, by  Theorem \ref{thm:constraintqualification}, we need to show that $\zc$ satisfies conditions \eqref{cond1}--\eqref{cond3}.  Conditions \eqref{eq:supportcondition} and \eqref{eq:complementaryslackness} are satisfied by construction.  As we noted in the proof of Theorem \ref{thm:disjointcliques_uniquerecovery}, condition \eqref{eq:traceone} 
is  taken care of by scaling.  As such, it remains to show that $Z'$ is PSD and that ${\rm range}(Z')={\rm ker}(\xc)$.

\paragraph{Simplification:}  We begin by showing that it suffices to prove the inequality $\sigma_{\max}(\zc - \zz) < \sigma_{\ks}(\zz)$ (as a reminder, $\sigma_{\ks}(\zz)$ is the smallest non-zero singular value of $\zz$).  To see why this is sufficient, let $\bv \in \mathbb{R}^{|\V|}$ and consider its direct sum decomposition with respect to $\mathcal{J}$; i.e., $\bv = \bv_{\mathcal{J}} + \bv_{\mathcal{J}^\perp}$.  Since $\zc \in \tilde{\mathcal{L}}$, we have by Proposition \ref{csexpnaded}, part (4) that   
$$
\bv^{T} \zc \bv = \bv^{T}_{\mathcal{J}} \zc \bv_{\mathcal{J}} = \bv^{T}_{\mathcal{J}} \zz \bv_{\mathcal{J}} + \bv^{T}_{\mathcal{J}} (\zc - \zz) \bv_{\mathcal{J}}.
$$
By Proposition \ref{thm:Unequalsizecanonicalmatrix_spectrum}, $\zz$ restricted on $\mathcal{J}$ is positive definite with smallest eigenvalue at least $\sigma_{\ks}(\zz)$, and hence 
$$
\bv^{T}_{\mathcal{J}} \zz \bv_{\mathcal{J}} \geq \sigma_{\ks}(\zz) \|\bv_{\mathcal{J}}\|_2^2.
$$
On the other hand, the inequality $\sigma_{\max}(\zc-\zz) < \sigma_{\ks}(\zz)$ implies
$$\bv^{T}_{\mathcal{J}} (\zc - \zz) \bv_{\mathcal{J}} <  \sigma_{\ks}(\zz) \|\bv_{\mathcal{J}}\|_2^2.$$
This means that $\bv^{T} \zc \bv \geq 0$ for all $\bv \in \mathbb{R}^{|\V|}$, which means that $\zc$ is PSD.  
  
Finally, we need to show that  ${\rm range}(Z')={\rm ker}(\xc)$. By definition, we have that $Z'\in \tilde{\mathcal{L}}$, so by Proposition \ref{csexpnaded}, we have that 
$$ {\rm range}(Z')\subseteq \mathcal{J}={\rm ker}(\xc).$$
For the converse inequality, we show that ${\rm ker}(Z')\subseteq \mathcal{J}^\perp$. For this, take $\bv\in {\rm ker}(Z')$. Let  $\bv = \bv_{\mathcal{J}} + \bv_{\mathcal{J}^\perp}$ and assume  that $\bv_{\mathcal{J}} \ne 0$. Then,  we would have
$$0=\bv^{T} \zc \bv =\bv^{T}_{\mathcal{J}} \zc \bv_{\mathcal{J}}>0,$$
leading to a contradiction.

\paragraph{Expressing  $\zz - \zc$ via the KKT conditions.} 
The first-order optimality condition corresponding to \eqref{eq:euclideanprojection} gives 
\begin{equation} \label{eq:foc}
\zz - \zc = \tilde{L}  + \tilde{K},
\end{equation}
where $\tilde{L}$ lies in the normal cone of $\tilde{\mathcal{L}}$ at $Z'$ and $\tilde{K}$ lies in the normal cone of $\mathcal{K}$ at $Z'$. Here 
we used the fact  that the normal cone to an intersection is the sum of the normal cones, e.g. see \cite[Corollary 23.8.1]{rockafellar}.  Moreover, as  the normal cone to a subspace is the orthogonal subspace, we have that  $\tilde{L} \in {\tilde{\mathcal{L}}}^\perp$ and $\tilde{K}\in \tilde{\mathcal{K}}^\perp$. First, by projecting \eqref{eq:foc} onto $\tilde{\mathcal{L}}$ we have
\begin{equation} \label{eq:projection_v1}
\zz - \zc = P_{\tilde{\mathcal{L}}}(\tilde{K}).
\end{equation}
This follows as  $\zz \in \tilde{\mathcal{L}}$, $\zc \in \tilde{\mathcal{K}} \cap \tilde{\mathcal{L}}$.  Second, by projecting \eqref{eq:foc} onto the subspace $\tilde{\mathcal{K}}^\perp$ we have
\begin{equation} \label{eq:projection_v2}
P_{\tilde{\mathcal{K}}^\perp} (\zz)  = P_{\tilde{\mathcal{K}}^\perp}(\tilde{L}) + \tilde{K}.
\end{equation}
Combining  \eqref{eq:projection_v1} and \eqref{eq:projection_v2}, we have
\begin{equation} \label{eq:zhatx_firstbound}
\| \zz - \zc \|_F = \| P_{\tilde{\mathcal{L}}}(\tilde{K}) \|_F \leq \| \tilde{K} \|_F =  \| P_{\tilde{\mathcal{K}}^\perp} (\zz) - P_{\tilde{\mathcal{K}}^\perp}(\tilde{L}) \|_F \leq \| P_{\tilde{\mathcal{K}}^\perp} (\zz) \|_F + \| P_{\tilde{\mathcal{K}}^\perp}(\tilde{L}) \|_F,
\end{equation}
where  the first equality follows from \eqref{eq:projection_v1}, and the last  inequality follows from  the triangle inequality.  

We next proceed to bound the terms $\|P_{\tilde{\mathcal{K}}^\perp} (\zz) \|_F$ and $  \| P_{\tilde{\mathcal{K}}^\perp}(\tilde{L}) \|_F$.

\paragraph{ Bound  the term $  \| P_{\tilde{\mathcal{K}}^\perp} (\zz) \|_F $.} We show that

\begin{equation}\label{csvdgb}
\| P_{\tilde{\mathcal{K}}^\perp}(Z^*) \|_F \leq \sqrt{c} \left(\sum_l 1/|\ccs_l|\right).
\end{equation}
The $(i,j)$-th entry of the matrix $P_{\tilde{\mathcal{K}}^\perp}(Z)$ is equal to $Z_{i,j}$ if $(i,j) \in \E$, and is equal to zero otherwise.  Consider the block corresponding to $(\ccs_x,\ccs_y)$, where $x\neq y$.  Each non-zero entry is $1/(|\ccs_x||\ccs_y|)$ and there are at most $c |\ccs_x||\ccs_y|$ entries.  Hence the sum of squares of the entries in this block is at most $c / (|\ccs_x||\ccs_y|)$.  We sum over the all indices $x$ and $y$ to obtain $\| P_{\tilde{\mathcal{K}}^\perp}(Z) \|_F^2  \leq \sum_{x,y} c / (|\ccs_x||\ccs_y|) \leq c (\sum_l 1/|\ccs_l|)^2$, from which the result follows.  (In the first inequality, recall that the block-diagonal entries of $P_{\tilde{\mathcal{K}}^\perp}(Z)$ are zero and thus do not contribute to the sum.)

\paragraph{Bound the term  $ \| P_{\tilde{\mathcal{K}}^\perp}(\tilde{L}) \|_F$.}  Setting   
 $$\epsilon = \langle \tilde{L} / \|\tilde{L}\|_F, \tilde{K} / \|\tilde{K}\|_F \rangle,$$
   it follows  from \eqref{eq:foc} that  
$$\| \zz - \zc \|_F^2 = \| \tilde{L} \|_F^2 + \| \tilde{K} \|_F^2 + 2 \epsilon\| \tilde{L} \|_F  \| \tilde{K} \|_F .$$  By the AM-GM inequality we have that 
$$ 2\| \tilde{L} \|_F   \| \tilde{K} \|_F\ge -\| \tilde{L} \|_F^2 - \| \tilde{K} \|_F^2.$$
Consequently, we get  
$$\| \zz - \zc \|_F^2 \geq (1-|\epsilon|)(\| \tilde{L} \|_F^2 + \| \tilde{K} \|_F^2)$$ and since $|\epsilon|\le 1$ we have 
\begin{equation}\label{csdvdfb}
\| \tilde{L} \|_F^2 \le \| \tilde{L} \|_F^2 + \| \tilde{K} \|_F^2 \leq \left(\frac{1}{1-|\epsilon|}\right) \| \zz - \zc \|_F^2.
\end{equation}
Combining \eqref{csdvdfb} with Theorem \ref{thm:incoherence_1} $ (ii)$   we have
\begin{equation}\label{xsdvdfv}
\| P_{\tilde{\mathcal{K}}^\perp}(\tilde{L}) \|_F \leq (2\sqrt{c})^{1/2} \|\tilde{L}\|_F \leq { (2\sqrt{c})^{1/2} \over \sqrt{1-|\epsilon|}} \| \zz - \zc \|_F.
\end{equation}

\paragraph{Concluding the proof.} 
We have established in \eqref{eq:zhatx_firstbound} that 
$$ \| \zz - \zc \|_F  \leq \| P_{\tilde{\mathcal{K}}^\perp} (\zz) \|_F + \| P_{\tilde{\mathcal{K}}^\perp}(\tilde{L}) \|_F,$$
which combined with \eqref{csvdgb} and \eqref{xsdvdfv}  shows that
\begin{equation}\label{cxvdbdf}
\| \zz - \zc \|_F \leq \sqrt{c} (\sum_l 1/|\ccs_l|) + { (2\sqrt{c})^{1/2}\over \sqrt{1-|\epsilon|}} \| \zz - \zc \|_F.
\end{equation}
By Theorem \ref{thm:incoherence_1} $(i)$ we have that 
$$|\epsilon| =|\langle \tilde{L} / \|\tilde{L}\|_F, \tilde{K} / \|\tilde{K}\|_F \rangle| \leq 2 \sqrt{c},$$ so for $c\le 1/100$ we get 
$${ (2\sqrt{c})^{1/2}\over \sqrt{1-|\epsilon|}}  \leq { 1\over 2}.$$
 Then,   \eqref{cxvdbdf} implies  that 
$$\| \zz - \zc \|_F \leq 2 \sqrt{c} (\sum_l 1/|\ccs_l|).$$
 Finally, using that  the Frobenius norm is always greater than the spectral norm, we have
$$
\sigma_{\max}( \zz - \zc) \leq \| \zz - \zc \|_F \leq 2 \sqrt{c} (\sum_l 1/|\ccs_l|),
$$
which is strictly smaller than $\sigma_{\ks}(\zz)$ whenever $c < \frac{1}{4} ^2 (\sigma_{\ks}(\zz)/\sigma_{\max}(\zz))^{2} $.
\end{proof}

\section{Recovery for planted clique cover instances} \label{sec:randomgraphs}

We conclude Result \ref{thm:result1} by applying Hoeffding's inequality onto the conclusions of Result \ref{thm:result2} directly.

\begin{theorem} \label{thm:uneq_random_cond}
Let $G$ be a random planted clique cover instance defined on the cliques $\{ \ccs_l \}_{l=1}^{\ks}$  where we introduce edges between cliques with probability $p$.  If $p < c$ then $G$ is $c$-neighborly with probability greater than 
$$
1 - \sum_{i=1}^{\ks} {|\ccs_i| \sum_{j \neq i}\exp(-2|\ccs_j|(c-p)^2}).
$$
\end{theorem}

\begin{proof}
Let $X_{ia,jb}$ be a binary variable equal to 1 if vertices $a \in \ccs_i$ and $b\in \ccs_j$ are connected, and equal to 0 otherwise.  First we have $\mathbb{E}[\sum_{b \in \ccs_j}{X_{ia, jb}}] = \sum_{b \in \ccs_j}{\mathbb{E}[X_{ia, jb}]} = |\ccs_j|p$.  By Hoeffding's inequality we have
\begin{equation}\label{ineq:Hoeffding_U}
\mathbb{P} \Big(\sum_{b \in \ccs_j}{X_{ia, jb}}  \geq t_j + |\ccs_j|p \Big)  = \mathbb{P} \Big(\sum_{b \in \ccs_j}{X_{ia, jb}} - \mathbb{E}[\sum_{b \in \ccs_j}{X_{ia, jb}}] \geq t_j \Big) \leq \exp ( -2t_j^2 / |\ccs_j| ).
\end{equation}
Using a union bound we get that 
$$\mathbb{P} \Big( \bigcap _{i, a\in \ccs_i, j \ne i } \Big( \sum_{b}{X_{ia, jb}} <t_j + |\ccs_j|p \Big) \Big) \ge 1- \sum_i|\ccs_i|
\sum_{j\ne i}\exp(-2t_j^2/|\ccs_j|).$$
Finally, given that $p<c$, we set $t_j=|\ccs_j|(c-p)>0$ to get the desired result. 
\end{proof}

\section{Numerical comparison to alternative techniques } \label{subsec:related work}
In this section, we assess the effectiveness of the Lov\'asz theta function for recovering planted clique cover instances by comparing it to alternative techniques in the literature.  Specifically, we compare our method with the following approaches:

\begin{itemize} 
\item {\em Community Detection approach:}  We frame the clique cover problem as an instance of community detection, and use the method proposed by Oymak and Hassibi~\cite{OH:11}.

\item {\em $k$-Disjoint-Clique approach: } We consider the clique cover problem as an instance of the $k$-disjoint-clique problem, and apply the method presented by Ames and Vavasis \cite{AV:14}.

\item {\em Subgraph Isomorphism approach:}  We view the problem as an instance of  subgraph isomorphism, and apply the framework proposed by Candogan and Chandrasekaran \cite{CC:18}.
\end{itemize}


\paragraph{Community Detection/Graph clustering.}  
Oymak and Hassibi \cite{OH:11}  (see also \cite{CJSX:14}) present a  method for  detecting communities in an undirected  graph  $\mathcal{G} = (\mathcal{V},\E)$  by solving the 
following  SDP:
\begin{equation} \label{eq:sparselowrank_comdetect}
\begin{aligned}
\min_{S,L} ~~ \lambda \| S \|_{1} +  \| L \|_{*} \quad \mathrm{s.t.} \quad S+L = A, \quad 0 \leq L \leq \matrixones.
\end{aligned}
\end{equation}
Here     $A$ is  the adjacency matrix of $G$,  $\lambda>0$ is a tuning parameter, $\| X \|_1 = \sum_{i,j} |X_{i,j}|$ is the L1-norm over the matrix entries, while $\|X\|_* = \sum_i \sigma_i (X)$ is the matrix nuclear norm. 

  Let $(\hat{S},\hat{L})$ be the optimal solution to this SDP.   The intuition behind the 
  SDP \eqref{eq:sparselowrank_comdetect} 
is that it performs a matrix deconvolution in which the adjacency matrix $A$ is separated as the 
sum of a low-rank component $\hat{L}$ and a sparse matrix $\hat{S}$.  
The hope is that (possibly by tuning the parameter~$\lambda$) the low-rank component takes the form $\hat{L} = \sum_{l} \bone_{\ccs_l} \bone_{\ccs_l}^T$, where $\ccs_l$ are the indicator vectors for {\em disjoint} communities, while the sparse component $\hat{S}$ captures noise that is (hopefully)~sparse.  

The specific form in \eqref{eq:sparselowrank_comdetect} builds upon a long sequence of ideas, originating from the observation that the L1-norm is effective at inducing sparse structure, while the matrix nuclear norm is effective at inducing low-rank structure \cite{CanPla:11,RFP:10}.  Subsequent work proposed the formulation \eqref{eq:sparselowrank_comdetect} for deconvolving matrices as the linear sum of a sparse matrix and a low-rank matrix \cite{CSPW:11,WPMGR:09}; in particular, Chandrasekaran et. al. \eqref{eq:sparselowrank_comdetect} show that penalization with the L1-norm plus the nuclear-norm succeeds at recovering both components so long as these components satisfy a {\em mutual incoherence} condition.  The formulation in \eqref{eq:sparselowrank_comdetect} is a reasonable approach for obtaining clique covers because one can view cliques as clusters.  

\paragraph{$k$-Disjoint-Clique Approach.} In the $k$-disjoint-clique problem,  the goal is to find $k$ disjoint cliques (of any size) that cover as many nodes of the graph as possible.  It is worth noting that in this particular setting, the number of cliques $k$ is constant and is determined by the user as a parameter. Ames and Vavasis  \cite{AV:14}  introduce the following SDP for this problem:
\begin{equation} \label{eq:av}
\begin{aligned}
\max ~~ & ~~ \langle X, \matrixones \rangle \\
\mathrm{s.t.} ~~ & ~~ X \bone \leq \bone \\
& ~~ X_{i,j} = 0 \quad (i,j) \notin \E, i \neq j \\
& ~~ \langle X, I \rangle = k \\
& ~~ X \succeq 0.
\end{aligned}
\end{equation}

\paragraph{Subgraph isomorphism.}  Candogan and Chandrasekaran \cite{CC:18} investigate the problem of finding {\em any} subgraph embedded within a larger graph, known as the {\em subgraph isomorphism problem}.  A natural approach is to search over matrices obtained by conjugating the adjacency matrix of the subgraph with the permutation group, leading to a quadratic assignment problem instance.  The basic idea in \cite{CC:18} is to relax this set by conjugating with the orthogonal group $O(|V|)$ instead.  The resulting convex hull is known as the Schur-Horn orbitope \cite{SSS:11}, and it admits a tractable description via semidefinite programming \cite{DW:09}.  In our context, the adjacency matrix of the subgraph (technically not a proper subgraph) is $\sum_{l} \bone_{\ccs_l} \bone_{\ccs_l}^T$.  Assuming that all cliques $\ccs_l$ have equal number of vertices $n$, the proposed SDP formulation for finding the sub-graph is \footnote{The above formulation can be further simplified by eliminating $Z_0$, and replacing the equality condition $Z_0+Z_1 = I$ with $Z_1 \preceq I$.  We retain $Z_0$ in \eqref{eq:sh} to be consistent with the formulation in \cite{CC:18}. }
\begin{equation} \label{eq:sh}
\begin{aligned}
\max ~~ & ~~ \langle A, X \rangle \\
\mathrm{s.t.} ~~ & ~~ X_{i,j} = 0 \quad (i,j) \notin \E, i \neq j \\
& ~~ X = 0 \cdot Z_0 + n \cdot Z_1 \\
& ~~ \langle Z_0, I \rangle = nk-k \\
& ~~ \langle Z_1, I \rangle = k \\
& ~~ Z_0 + Z_1 = I \\
& ~~ Z_0, Z_1 \succeq 0.
\end{aligned}
\end{equation}
There is one severe limitation of formulating the clique cover instance as a subgraph isomorphism problem -- one needs knowledge of the {\em spectrum} of the adjacency matrix of the clique cover.  This assumption is certainly unrealistic, though one should keep in mind that Candogan and Chandrasekaran's work in \cite{CC:18} was designed to solve a completely different problem in the first place.

\subsection{Experimental setup and results}

In our experimental setup, we generate  clique cover  instances from the planted clique cover model  with $\ks=10$ cliques, each of size $n=10$.  For each value of   $p \in \{ 0.00, 0.05, 0.10, \ldots, 1.00 \}$, we generate $20$ random clique cover  instances.  We compare the Lov\'asz theta against the three methods described in the previous section. Our result are summarized in  Figure \ref{fig:comparison},  where we illustrate the average success rates of each method in recovering the underlying clique cover instances across the entire spectrum of $p$ values. In  the reported experiments we use two notions of recovery for $\lt$:
\begin{enumerate}
\item We declare {\em strong recovery}  when the solver outputs an optimal solution $\hat{X}$ that is equal to $\xc$ (up to a small tolerance).
\item We declare {\em weak recovery} when $\lt(G) = \ks$ but the output solution $\hat{X}$ is not equal to $\xc$. \end{enumerate}

First, we  compare against  the SDP   \eqref{eq:sparselowrank_comdetect} for community detection.  We sweep over the choices regularization parameters $\lambda \in \{ 0.0, 0.1, \ldots, 1.0 \}$.    This sweeping process makes their method the most expensive to implement among all methods listed here. For a given choice of  $\lambda$, let $(\hat{L},\hat{S})$ denote the optimal solution.  We say that the method succeeds if $\hat{L} = \sum \bone_{\ccs_l}\bone_{\ccs_l}^T$ for some choice of parameter~$\lambda$. 

Second,  we compare against  the SDP \eqref{eq:av} for the $k$-disjoint clique problem.  We say that the method succeeds if the optimal solution satisfies $\hat{X} = \sum \bone_{\ccs_l}\bone_{\ccs_l}^T$. It's worth noting that when applying this method to uncover clique covers, there is a minor constraint in that we must specify the size of cliques $\ks=10$, or alternatively, sweep over all feasible clique sizes.

Our third basis of comparison is the SDP \eqref{eq:sh} proposed by Candogan and Chandrasekaran for finding subgraphs based on the Schur-Horn orbitope of a graph \cite{CC:18}.  We say that the method succeeds if the optimal solution satisfies $\hat{X} = \sum \bone_{\ccs_l}\bone_{\ccs_l}^T$. A limitation of using this  
approach  for finding a clique cover is that we  implicitly assume that we know what we are searching for in the first place (to be precise, it is the {\em spectrum} of the target graph we require as an input).  This assumption is somewhat unrealistic, but one should bear in mind that Candogan and Chandrasekaran's work in \cite{CC:18} was intended to solve a completely different problem in the first place.

Our experiments are summarized in Figure \ref{fig:comparison}. First, when examining the performance of the Lovász theta function $\lt$, we observe a congruence between our theoretical predictions, which predict  strong recovery for small $p$ values, and the outcomes derived from our numerical experiments. Furthermore, it is noteworthy that this strong recovery pattern persists  until approximately $p\approx 0.4$.

In the context of our comparisons with the other three methods, the Lovász theta function emerges as the most successful,  with only a slight edge over the SDP \eqref{eq:av} proposed by Ames and Vavasis \cite{AV:14} for finding disjoint cliques, as well as the Schur-Horn orbitope-based relaxation \eqref{eq:sh} presented by Candogan and Chandrasekaran. The method that exhibited the weakest performance in our experiments was the matrix deconvolution method by Oymak and Hassibi \cite{OH:11}.

\begin{figure}[H]
\centering
\includegraphics[width=0.5\textwidth]{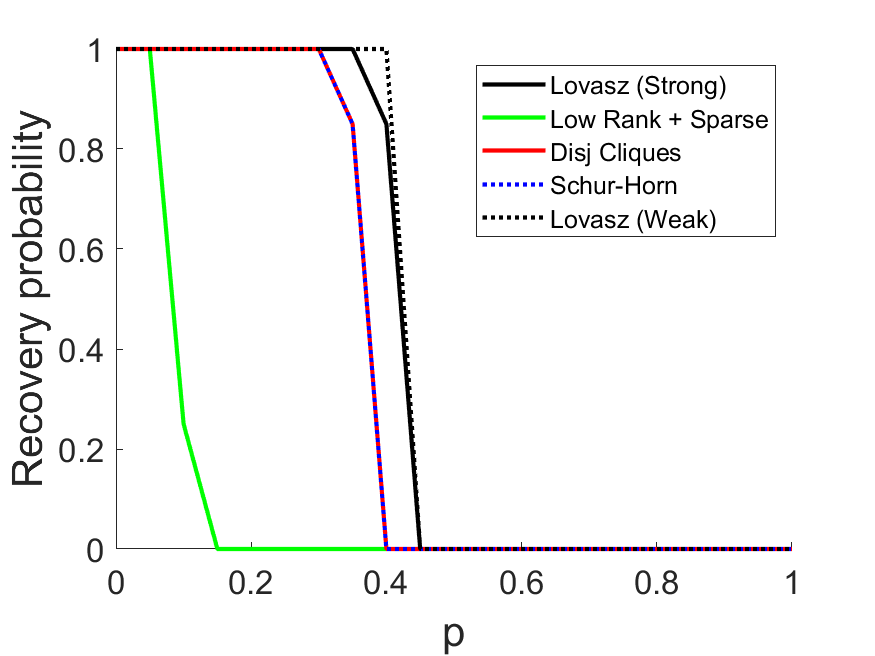}
\caption{Comparison of Lov\'asz theta function with other methods.}
\label{fig:comparison}
\end{figure}

\section{Future directions} \label{sec:numerical}

In this section, we conduct and discuss additional numerical experiments involving the Lov\'asz theta function $\lt(G)$.  Our objective is to report interesting behavior of the Lov\'asz theta function in the context of recovering clique covers that are not supported or apparent from our theoretical findings, and suggest future directions to investigate.

\subsection{Comparison with ILP}

In our first example, we compare how tight the Lov\'asz theta function is in comparison with the clique cover number.  To investigate this question, we generate a graph from the planted clique cover model for varying choices of $p \in [0,1]$.  We compute the clique cover number by solving the following Integer Linear Program (ILP)
\begin{equation} \label{eq:cliquecover_ilp}
\begin{aligned}
\min ~~ & ~~ \sum_{c} y_c \\
\mathrm{s.t.} ~~ & ~~ x_{i,c} \leq y_c \quad \text{ for all } \quad i,c \\
& ~~ x_{i,c} + x_{j,c} \leq 1 \quad (i,j) \in \E \\
& ~~ \sum_{c} x_{i,c} = 1 \quad \text{ for all } \quad i \\
& ~~ x_{i,c}, y_c \in \{0,1\}.
\end{aligned}
\end{equation}
Here, the binary variable $x_{i,c}$ indicates whether vertex $i$ is contained in clique $c$, while the binary variable $y_c$ indicates whether clique $c$ is non-empty.  The constraint $x_{i,c} + x_{j,c} \leq 1$ ensures no adjacent pair of vertices are in the same clique, while the constraint $\sum_{c} x_{i,c} = 1$ ensures that every vertex belongs to a unique clique. 

Our motivation for studying this question is as follows:  In our experiments in Section \ref{subsec:related work}, we generally take the clique covering number to be equal to $\ks$.  However, short of solving computing the clique covering number outright (say, via \eqref{eq:cliquecover_ilp}), there is no simple way of verifying this is indeed the case.  The concern increases as $p$ increases, as we may inadvertently lower the clique covering number when more edges are added.  As such, one objective of this experiment is to understand if our assumption that the clique covering number is well approximated by $\ks$ is sound.  

In the left sub-plot of Figure \ref{fig:ILP} we compare the clique covering number with the Lov\'asz theta function on a planted clique cover instance with $8$ cliques, each of size $8$, and with varying choices of parameter $p$.  Based on our results we notice that the deviation between these quantities is at most $\approx 2$, and is greatest at $p \approx 0.7$.  

In the right sub-plot of Figure \ref{fig:ILP} we compare the time taken for the ILP solver Gurobi \cite{gurobi} to compute \eqref{eq:cliquecover_ilp} with the SDP solver SDPT3 \cite{sdpt3:1,sdpt3:2} to compute \eqref{eq:lovasz_lambdamax}.  The time taken for SDP solver is approximately equal across all problem instances, as one would expect.  The time taken for the ILP solver is quite small for $p \leq 0.4$, but becomes substantially longer for $p=0.6$.  In the same plot we also track the number of simplex solves required by the ILP solver Gurobi -- the trend of this curve is largely similar to the previous curve.  These observations suggest that the most complex regime for computing the clique covering number for the Erd\H{o}s-R\'enyi noise model occurs around $p \approx 0.6$.  There are a number of explanations: First, our theoretical findings suggest that it is generally easier for the Lov\'asz theta function to discover minimal clique covers for small values of $p$, and these curves do also suggest that small values of $p$ correspond to `easier' instances.  On the other extreme, at $p = 1$, there is only one clique.  For values of $p = 1-\epsilon$, it may be that large cliques are easy to find, and hence ILP solvers are able to certify optimality relatively quickly.  It would be interesting to investigate the behavior of these curves for increasing problem sizes, though one obvious difficulty is that the ILP instances may become impractical to compute.  As a note, the SDP experiments were performed using the SDPT3 solver \cite{sdpt3:1,sdpt3:2}, called from the CVX parser \cite{cvx,gb08} with default precision setting and running on MATLAB R2022a.  The ILP experiments were conducted with Gurobi version 10.0.1 \cite{gurobi}.

\begin{figure}[H]
\centering
\includegraphics[width=0.48\textwidth]{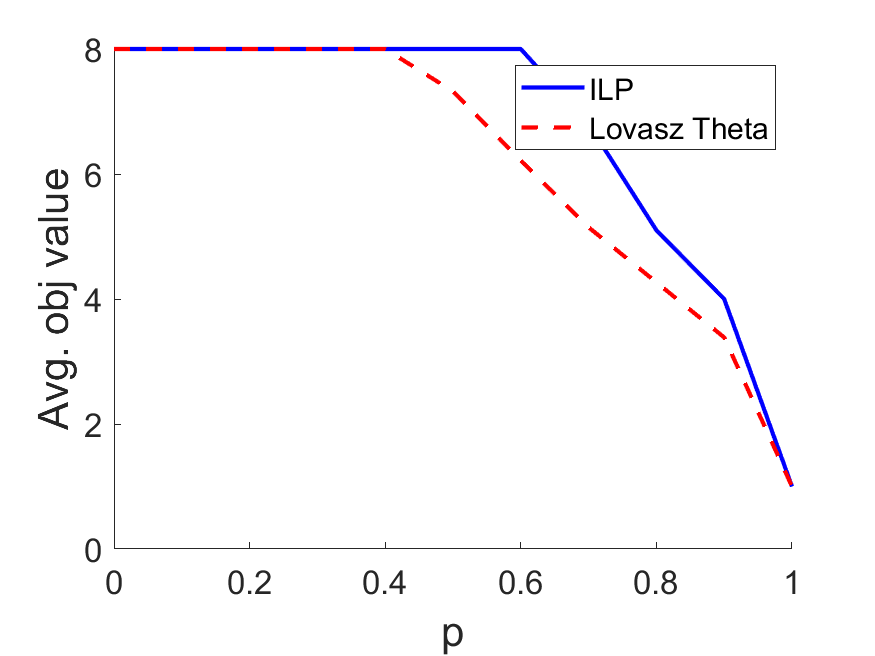}
\includegraphics[width=0.48\textwidth]{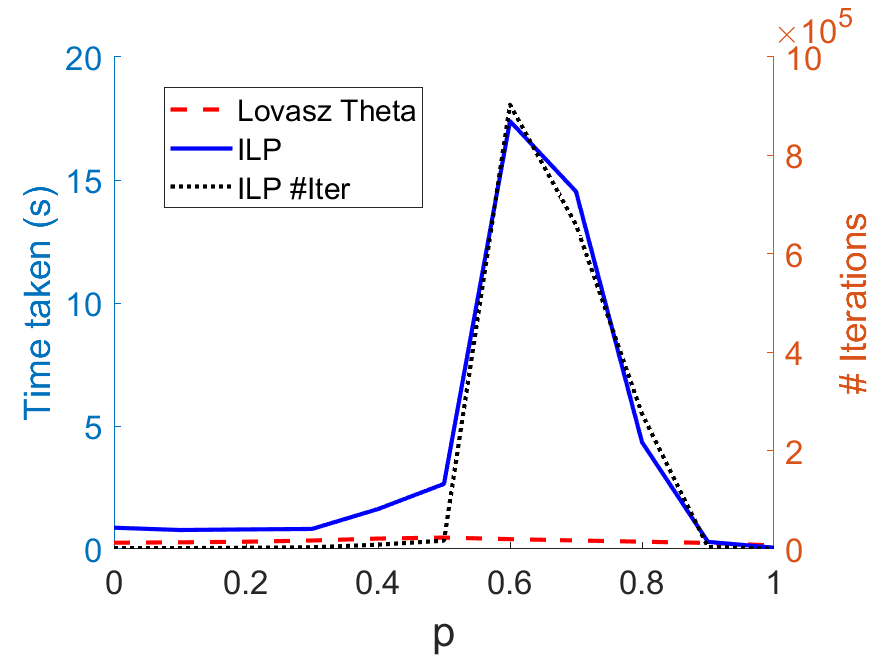}
\caption{Comparison of the clique covering number with Lov\'asz theta for random clique cover model (left).  Comparison of time taken for ILP solver to compute clique covering number with time taken for SDP solver to compute Lov\'asz theta function (right). The number of simplex iterations taken by ILP solver is included as a reference.}
\label{fig:ILP}
\end{figure}

\subsection{Phase Transition}

In the second experiment, we investigate the probability that the Lov\'asz theta function correctly recovers the planted cliques as certain parameters are taken to $+\infty$.  Our objective is to see if a phase transition arises.

To investigate this problem, we consider an experimental set-up where all the cliques are of equal size.  We set the number of cliques $\ks$ to be equal to the size of each clique $n$, with the values $n = \ks$ ranging over $\{5,\ldots, 15\}$.  For each value of $n$ and $\ks$, we generate $10$ random graphs from the planted clique cover model with $p$ taking values from $\{ 0.00, 0.05, 0.10, \ldots, 1.00 \}$.  In Figure \ref{fig:phasetransition} we plot the success probabilities corresponding to each value of $p$.  We plot the curve corresponding to strong recovery in black, and the curve corresponding to weak recovery in red.

We make a number of observations.  First, we observe that the transition from success to failure becomes more abrupt as we increase the values of $n$ and $\ks$.  This observation holds for the both types of recovery conditions.  In particular, our results suggest that a phase transition occurs as these values are taken to $+\infty$, and we conjecture that this is indeed the case, as is typical with many phenomena in random graphs.  Second, and quite interestingly, the gap between both types of recovery {\em shrink} as the problem parameters $n$ and $\ks$ increase.  We conjecture that both curves {\em coincide} in the limit.

\begin{figure}[H]
\centering
\includegraphics[width=0.5\textwidth]{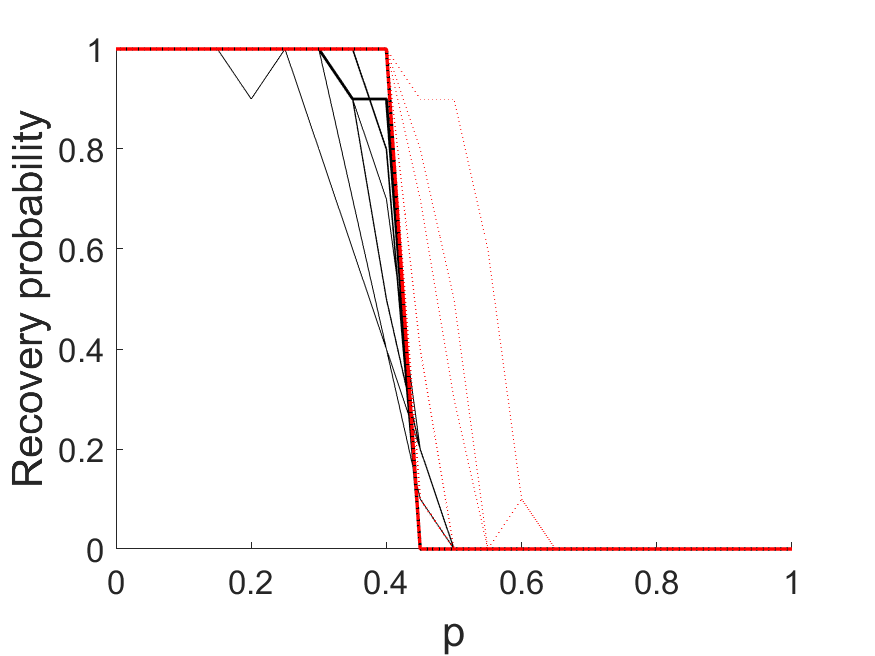}
\caption{Probability of correctly recovering a planted clique instance with increasing clique size and increasing number of cliques.  The black curves correspond to strong recovery while the red curves correspond to weak recovery.  The thickness of the lines increase with the problem parameters $n$ and $\ks$.}
\label{fig:phasetransition}
\end{figure}

\paragraph{Acknowledgements.} YS gratefully acknowledges Ministry of Education (Singapore) Academic Research Funds (Tier 1) R-146-000-329-133 and A-8002497-00-00. AV is supported by the MOE Tier 2 Grant (MOE-T2EP20223-0018), the National Research Foundation, Singapore, under its QEP2.0 programme (NRF2021-QEP2-02-P05), the CQT++ Core Research Funding Grant (SUTD) (RS-NRCQT-00002), the National Research Foundation Singapore and DSO National Laboratories under the AI Singapore Programme (Award Number: AISG2-RP-2020-016), and partially by Project MIS 5154714 of the National Recovery and Resilience Plan, Greece 2.0, funded by the European Union under the NextGenerationEU Program.

\bibliography{bib_kcliques}

\end{document}